\documentclass[a4paper,reqno]{amsart}

\usepackage{geometry}
\title{Tangent spaces of diffeological spaces and their variants}
\author{Taho Masaki}

\makeatletter
\@namedef{subjclassname@2020}{\textup{2020} Mathematics Subject Classification}
\makeatother


\address{GRADUATE SCHOOL OF MATHEMATICAL SCIENCES, THE
UNIVERSITY OF TOKYO, 3-8-1 KOMABA, MEGURO-KU, TOKYO, 153-8914,
JAPAN}
\email{taho@ms.u-tokyo.ac.jp}

\usepackage{amsmath} 
\usepackage{amssymb} 
\usepackage{amsfonts} 
\usepackage{mathrsfs} 
\usepackage{amsthm} 
\usepackage{bm} 
\usepackage[new]{old-arrows} 
\usepackage[dvipdfmx]{graphicx} 
\usepackage{wrapfig} 
\usepackage{color} 
\usepackage[labelsep=colon]{caption} 
\usepackage[labelformat=empty,subrefformat=parens]{subcaption} 
\usepackage{multicol} 
\usepackage{units} 
\usepackage[all]{xy} 
\usepackage{tikz-cd} 
\usepackage{hyperref} 

\usepackage{cleveref}

\usepackage{thmtools}
\usepackage{thm-restate}  

\newtheorem{thm}{Theorem}[section]
\newtheorem{prop}[thm]{Proposition}
\newtheorem{lemma}[thm]{Lemma}
\newtheorem{cor}[thm]{Corollary}

\newtheorem*{mthm*}{Main Theorem}
\newtheorem*{thm*}{Theorem}
\newtheorem*{prop*}{Proposition}

\theoremstyle{definition}
\newtheorem{defi}[thm]{Definition}
\newtheorem{ex}[thm]{Example}
\newtheorem{remark}[thm]{Remark}

\newcommand{\topsp}{{\mathrm {Top}}}                 
\newcommand{\mfd}{{\mathrm {Mfd}}}                   
\newcommand{\dflg}{{\mathrm {Dflg}}}                 
\newcommand{\encl}{{\mathrm {Eucl}}}                 
\newcommand{\dflgbased}{{\mathrm {Dflg}}_*}          
\newcommand{\enclbased}{{\mathrm {Eucl}}_*}          
\newcommand{\enclzero}{{\mathrm {Eucl}}_0}           
\newcommand{\dflgloc}{{\mathrm {Dflg}}_*^{loc}}      
\newcommand{\enclloc}{{\mathrm {Eucl}_*^{loc}}}      
\newcommand{\enclzeroloc}{{\mathrm {Eucl}_0^{loc}}}  
\newcommand{\vect}{{\mathrm {Vect}}}                 
\newcommand{\dvect}{{\mathrm {DVect}}}               
\newcommand{\vsdcat}{{\mathrm {VSD}}}
\newcommand{\dvscat}{{\mathrm {DVS}}}

\newcommand{\coplotcat}{(X,x)\downarrow\inclloc}

\newcommand{\dfrn}{\mathrm {Dfrn}}

\newcommand{\incl}{i}                                
\newcommand{\inclbased}{i_*}
\newcommand{\inclzero}{i_0}
\newcommand{\inclloc}{i_*^{loc}}                       
\newcommand{\inclzeroloc}{i_0^{loc}}                 
\newcommand{\encltoloc}{\mathop{Germ}}               
\newcommand{\dflgtoloc}{\mathop{Germ}^{D}}
\newcommand{\conti}{C}

\newcommand{\equivrel}{\mathord{\sim}}               
\newcommand{\dom}{\operatorname{dom}}            
\newcommand{\plotdom}[1]{U_{#1}}                      
\newcommand{\coplotdom}[1]{B_{#1}}
\newcommand{\coplotim}[1]{V_{#1}}
\newcommand{\Lan}{\operatorname{Lan}}                   
\newcommand{\Ran}{\operatorname{Ran}}                   
\newcommand{\ad}{\operatorname{ad}}              
\newcommand{\obj}{\operatorname{Obj}}            
\newcommand{\prl}{\operatorname{pr}_1}
\newcommand{\prr}{\operatorname{pr}_2}
\newcommand{\plus}{\operatorname{+}}
\newcommand{\kakeru}{\operatorname{\times}}
\newcommand{\id}{\operatorname{id}}
\newcommand{\nattrans}{\operatorname{\eta}^R}
\newcommand{\nattransexternal}{\operatorname{\eta}}
\newcommand{\precomp}[1]{\circ {#1}}

\newcommand{\del}[1]{{\partial}_{#1}}
\newcommand{\nattransglobal}{\operatorname{\mu}^R}
\newcommand{\nattransglobalexternal}{\operatorname{\mu}}

\newcommand{\hectorbundle}{T^H}
\newcommand{\dvsbundle}{T^{dvs}}

\newcommand{\cross}{\text{{\dag}}}                       
\newcommand{\crosssub}{\cross_{sub}}                 
\newcommand{\crossquot}{\cross_{quot}}               
\newcommand{\halfsub}{H_{sub}}                       
\newcommand{\irrationaltorus}{T_{\theta}^2}          
\newcommand{\locmap}{{\mathrm {Localmap}}}           
\newcommand{\germ}{G}                                
\newcommand{\plots}{{\mathrm {Plots}}}                           

\newcommand{\internalfunctor}{T}                     
\newcommand{\internal}[2]{T_{#1} (#2)}               
\newcommand{\rightfunctor}{\hat{T}^R}
\newcommand{\righttangent}[2]{\hat{T}^R_{#1} (#2)}
\newcommand{\externalfunctor}{\hat{T}}
\newcommand{\external}[2]{\hat{T}_{#1} (#2)}
\newcommand{\vincenttangent}[2]{\hat{T}^V_{#1} (#2)}

\newcommand{\hectorinternal}[2]{T^{H}_{#1} (#2)}
\newcommand{\dvsinternal}[2]{T^{dvs}_{#1} (#2)}

\newcommand{\globalrighttangent}[2]{\hat{\mathbb{T}}^R_{#1} (#2)}
\newcommand{\globalexternal}[2]{\hat{\mathbb{T}}_{#1} (#2)}
\newcommand{\globalrightfunctor}{\hat{\mathbb{T}}^R}
\newcommand{\globalexternalfunctor}{\hat{\mathbb{T}}}
\newcommand{\globalvincenttangent}[2]{\hat{\mathbb{T}}^V_{#1} (#2)}
\newcommand{\smoothmap}[2]{C^{\infty}(#1,#2)}
\newcommand{\smoothlinear}[2]{L^{\infty}(#1,#2)}
\newcommand{\linear}[2]{L(#1,#2)}
\newcommand{\gl}[1]{GL(#1)}

\newcommand{\regular}{regular~}
\newcommand{\frolicher}{Fr\"{o}licher~}

\newcommand{\IZ}{\cite{MR3025051}}
\newcommand{\baez}{\cite{6db4c29e-8954-3d52-84fa-c44d21fd5d68}}
\newcommand{\CW}{\cite{MR3467758}}
\newcommand{\vincent}{\cite{Vincent2008DiffeologicalDG}}
\newcommand{\watts}{\cite{MR3153238}}
\newcommand{\BKW}{\cite{batubenge2023diffeological}}
\newcommand{\SH}{\cite{MR3884529}}
\newcommand{\wu}{\cite{MR3350086}}
\newcommand{\DVS}{\cite{MR4044857}}

\newcommand{\dtop}{\cite{MR3270173}}
\newcommand{\souriau}{\cite{MR0607688}}
\newcommand{\maclane}{\cite{MR1712872}}
\newcommand{\nonnegativeroot}{\cite{MR2200169}}

\newcommand{\italic}[1]{{\it #1}}

\newcommand{\thref}[1]{Theorem~\ref{#1}}
\newcommand{\proref}[1]{Proposition~\ref{#1}}
\newcommand{\defref}[1]{Definition~\ref{#1}}
\newcommand{\exref}[1]{Example~\ref{#1}}
\newcommand{\lemref}[1]{Lemma~\ref{#1}}
\newcommand{\corref}[1]{Corollary~\ref{#1}}
\newcommand{\remref}[1]{Remark~\ref{#1}}

\begin{document}

\begin{abstract}
  Several methods have been proposed to define tangent spaces for diffeological spaces. Among them, the internal tangent functor is obtained as the left Kan extension of the tangent functor for manifolds.
  However, the right Kan extension of the same functor has not been well-studied. In this paper, we investigate the relationship between this right Kan extension and the external tangent space, another type of tangent space for diffeological spaces.
  We prove that by slightly modifying the inclusion functor used in the right Kan extension, we obtain a right tangent space functor, which is almost isomorphic to the external tangent space.
  Furthermore, we show that when a diffeological space satisfies a favorable property called smoothly regular, this right tangent space coincides with the right Kan extension mentioned earlier. 
\end{abstract}

\maketitle

\section{Introduction} \label{intro}
Diffeological spaces were introduced by Souriau \souriau~in the 1980s. They are one of the generalizations of smooth manifolds, which include manifolds with corners, orbifolds, and infinite dimensional manifolds. 
Moreover, diffeological spaces have a lot of types of spaces because diffeological spaces are closed under several operations, such as taking quotient spaces, subspaces, mapping spaces, etc. 
In fact, the category of diffeological spaces is complete, cocomplete, and cartesian closed (Baez-Hoffnung \baez). 
This is also explained in the standard textbook for diffeology by Iglesias-Zemmour \IZ, and we briefly summarize basic properties in Section \ref{section diffeological space}. 
As generalizations of smooth manifolds, differential spaces, and \frolicher spaces are also known. Their definitions and properties are summarized in the Appendix \ref{section differential and frolicher}. 
This appendix is a summarization of Watts \watts~and Batubenge-Karshon-Watts \BKW. 
Familiarity with the definitions and basic properties of these concepts provide better insights when examining the properties of tangent spaces of diffeological spaces and their variants. 
However, since these concepts are not directly essential to the main focus of this paper, they are placed in the appendix. 

Tangent spaces for diffeological spaces have been studied mainly by Vincent \vincent, \IZ~and Christensen-Wu \CW. 
Vincent defined a tangent space of a diffeological space as a vector space of smooth curves divided by an equivalence relation whether the induced smooth derivatives are equal or not (\vincent, Chapter 3). 
Iglesias-Zemmour defined a tangent space of a diffeological space as a subspace of the dual space of the space of 1-forms (\IZ, 6.53), 
where differential forms on diffeological spaces are defined in \IZ, 6.28, without using the notion of tangent spaces. 
Christensen-Wu defined two types of tangent spaces. One of them, which is called the internal tangent space, was defined as 
a vector space of velocity vectors of smooth curves 
(\CW, Definition 3.1). 
The other one is called the external tangent space, which was defined as a vector space of smooth derivations on germs of smooth real-valued functions (\CW, Definition 3.10). 

In \CW, Definition 3.1, it was explained that the internal tangent functor is the left Kan extension of the tangent functor from the category $\enclzero$, 
which is the category of open sets of Euclidean spaces with the origin and based smooth maps between them (defined in \defref{definition euclidean category}). 
On the other hand, the right Kan extension of the same functor remains largely unexplored. 
In the Section \ref{section external}, we investigate the relationship between the external tangent space $\external{x}{X}$ at the point $x$ of a diffeological space $X$ and the right Kan extension. 
To discuss the relationship, we define the category $\enclloc$ in \defref{definition euclidean category} which has the same objects as $\enclzero$ and its morphisms are germs of smooth maps. 
Also, we define the right tangent space $\righttangent{x}{X}$ (\defref{definition external}), which is almost isomorphic to the external tangent space 
(I do not know of any example such that the external tangent space and the right tangent space do not coincide).
In \thref{theorem right tangent space is right Kan extension}, we prove that 
the right tangent space is isomorphic to the image of the right Kan extension of the tangent functor from the category $\enclloc$, not $\enclzero$ 
(the image of the right Kan extension of the tangent functor from $\enclzero$ is called the global right tangent space denoted by $\globalrighttangent{x}{X}$, defined in \defref{definition other external tangent space}). 
\begin{thm*}{\rm\bf\ref{theorem right tangent space is right Kan extension}} \label{a}
  Let $X$ be a diffeological space, and let $x\in X$. Then there is an isomorphism of vector spaces:
  \[
    \Ran_{\inclloc} T^{\enclloc} (X,x)\cong \righttangent{x}{X}.
  \]
\end{thm*}
In this theorem, the left side of the isomorphism is the image of the right Kan extension of the tangent functor from $\enclloc$.
See Sections \ref{section internal tangent spaces of diffeological space} and \ref{section external} for details.
By this result, we sometimes simplify the calculations of the right or external tangent spaces. 

Also, we consider the relationship between the right or external tangent spaces and their global versions 
(\defref{definition other external tangent space}, the definition uses global smooth functions, not germs of smooth functions). 
In \thref{theorem global right tangent space of regular space} and \thref{theorem global external tangent space of regular space}, 
we prove that the standard versions of right or external tangent spaces are isomorphic to their global versions if the diffeological space is smoothly regular (\defref{definition smoothly regular}). 
\begin{thm*}{\rm\bf\ref{theorem global right tangent space of regular space}} \label{b}
  Let $X$ be a diffeological space, and let $x\in X$. If $X$ is smoothly regular at $x$, then there is an isomorphism of vector spaces:
  \[
    \globalrighttangent{x}{X}\cong\righttangent{x}{X}.
  \]
\end{thm*}
\begin{thm*}{\rm\bf\ref{theorem global external tangent space of regular space}} \label{c}
  Let $X$ be a diffeological space, and let $x\in X$. If $X$ is smoothly regular at $x$, then there is an isomorphism of vector spaces:
  \[
    \globalexternal{x}{X}\cong\external{x}{X}.
  \]
\end{thm*}
Here, $\globalexternal{x}{X}$ is the global external tangent space of a diffeological space $X$ at the point $x\in X$.
By these results, we can calculate the right or external tangent spaces of many diffeological spaces by calculating their global versions instead.
In many cases, this simplifies the calculations of tangent spaces. 
We have also constructed an example of a diffeological space that is not smoothly regular, in which the isomorphisms in these theorems fail to hold (\exref{example right tangent not equal to global right tangent}).

Moreover, we see some cases in which it is difficult to determine the right or external tangent spaces completely, but it is possible to determine the tangent space of Vincent (\vincent, Chapter 3). 
One of the most interesting examples in this paper is \exref{example external tangent eigenvalue>,<1}. 
In this example, to investigate the right tangent space of the diffeological space $X$ of this example, we think of a diffeological space $Y$ instead of $X$, where $Y$ is similar to $X$ but a little different. 
We only know that the global right tangent space of $X$ at the point $[(0,0)]$ is more or equal to 1-dimensional, but this seems to be a non-trivial and interesting fact. 

Throughout this paper, there is no diffeological space that is known to have non-isomorphic external tangent space and right tangent space. 
Finding this example is one of the unsolved problems.

In Section \ref{section internal tangent spaces of diffeological space}, we review the definitions and basic properties of the internal tangent space $\internal{x}{X}$ and the internal tangent bundle defined in \CW.
After that, we get the induced action on the internal tangent space of an orbit space by some action called a pointwise smooth action. 
\begin{prop*}{\rm\bf\ref{proposition induced linear map from internal tangent of quotient}} \label{d}
  Let $X$ be a diffeological space, let $x\in X$, let $G$ be a group, and let $\varphi$ be a pointwise smooth action of $G$ on $X$. Let $\pi\colon X\to X/G_x$ be the standard projection.
  The map $\internalfunctor(\pi)\colon\internal{x}{X}\to \internal{[x]}{X/G_x}$ induces a smooth linear map $(\internal{x}{X})_{G_x}\to\internal{[x]}{X/G_x}$, 
  where for a diffeological vector space $V$ and a smooth linear action of $H$ on $V$, we write $V_H=V/\langle v-h\cdot v\mid v\in V, h\in H\rangle$ for the $H$-coinvariant.
\end{prop*}
This proposition is useful for calculations of quotient diffeological spaces by some group actions.
Also, in some special cases, we know that the induced linear maps are automatically isomorphism (\corref{corollary internal tangent of quotient of fine vector space}). 
Our results offer a new perspective on the relationships between various methods of defining tangent spaces for diffeological spaces, contributing to a deeper understanding of this topic.

In this paper, all manifolds are finite-dimensional, Hausdorff, second countable, and without boundary. 
Also, we follow the categorical notation and terminology of Mac Lane \maclane.

\section*{Acknowledgement}
This work was supported by JSPS KAKENHI Grant Number 24KJ0881. 

I would like to express my gratitude to my supervisor, Takuya Sakasai for his invaluable advice and tremendous support. 

Also, I would like to express my gratitude to the WINGS-FMSP program for its financial support. 
In particular, I would like to express my gratitude to my supporting supervisor in the WINGS-FMSP program, Toshiyuki Kobayashi for his huge support and encouragement. 
Lastly, I would like to express my gratitude to Professor Katsuhiko Kuribayashi at Shinshu University for his generous advice and for many stimulating conversations at every research meeting.

\section{Diffeological spaces} \label{section diffeological space}

In this section, we define {\it diffeological spaces} and comment on their their basic properties.
All definitions and properties in this section are found in \IZ, the standard textbook of diffeology. 
We omit the proofs of propositions in this section. See \IZ. 

\begin{defi}[\IZ, 1.5] \label{definition diffeology}
  Let $X$ be a set. A \italic{parametrization} of $X$ is a map $U\to X$ where $U$ is an open set of $\mathbb{R}^n$ for some $n\in\mathbb{Z}_{\geq 0}$.
  A \italic{diffeology} on $X$ is a set of parametrizations $\mathscr{D}_X$ (we call an element of $\mathscr{D}_X$ plot) such that the following three axioms are satisfied:
  \begin{description}
    \item[Covering] Every constant parametrization $U\to X$ is a plot.
    \item[Locality] Let $p\colon U\to X$ be a parametrization. If there is an open covering $\{U_{\alpha}\}$ of $U$ such that ${p|}_{U_{\alpha}}\in \mathscr{D}_X$ for all $\alpha$, then $p$ itself is a plot.  
    \item[Smooth compatibility] For every plot $p\colon U\to X$, every open set $V$ in $\mathbb{R}^m$ and $f\colon V\to U$ that is smooth as a map between Euclidean spaces, $p\circ f\colon V\to X$ is also a plot.
  \end{description}
  A diffeological space is a set equipped with a diffeology.
  We usually write the underlying set $X$ to represent a diffeological space $(X,\mathscr{D}_X)$.
\end{defi}

Throughout this paper, we call an open set of $\mathbb{R}^n$ for some $n$ a \italic{Euclidean open set}, and we denote the domain set of a plot $p$ of a diffeological space by $\plotdom{p}$.
Also, we call a plot $p$ an \italic{$n$-plot} if its domain $\plotdom{p}$ is an open set of $\mathbb{R}^n$. 

\begin{defi}[\IZ, 1.14] \label{definition smooth map}
  Let $X$ and $Y$ be two diffeological spaces. 
  A map $f\colon X\to Y$ is said to be smooth if for every plot $p\colon \plotdom{p}\to X$, the map $f\circ p\colon \plotdom{p}\to Y$ is a plot of $Y$.
  The set of smooth maps from $X$ to $Y$ is denoted by $\smoothmap{X}{Y}$. 
  If $X$ is a Euclidean open set, $C^{\infty}(X,Y)$ is equal to the set of all plots $X\to Y$ (\IZ, 1.16).
\end{defi}

Diffeological spaces and smooth maps between them form a category denoted by $\dflg$.
We call an isomorphism of this category a \italic{diffeomorphism}. The category $\dflg$ contains the category $\mfd$ of smooth manifolds and smooth maps as a full subcategory. 
In particular, $\dflg$ also contains the category $\encl$ of Euclidean open sets and smooth maps as a full subcategory. This follows from:

\begin{ex} \label{example definition manifold diffeology}
  Let $M$ be a smooth manifold, and let $\mathscr{D}_M$ be the set of all smooth maps from some open set $U$ of Euclidean spaces to $M$ (in the meanings of smooth map between manifolds).
  Then $\mathscr{D}_M$ is a diffeology on $M$. 
\end{ex}

\begin{ex}[\IZ, 1.20, 1.21] \label{example definition max/min diffeology}
  Let $X$ be a set and $\mathscr{D}_X^{max}$ the set of all parametrizations of $X$. 
  Then $\mathscr{D}_X^{max}$ is a diffeology of $X$ and we call this diffeology the \italic{largest} diffeology on $X$.
  Let $\mathscr{D}_X^{min}$ be a set of all locally constant parametrizations of $X$. 
  Then $\mathscr{D}_X^{min}$ is a diffeology of $X$ and we call this diffeology the \italic{smallest} diffeology on $X$\footnote{Sometimes these diffeologies are called the coarse diffeology and the discrete diffeology respectively (for example, in \IZ). 
  These terms are based on concepts from the field of topology (like topological spaces, every map from the discrete diffeological space is smooth). However, I think that these terms are very confusing. 
  Therefore in this paper, we call these the largest diffeology and the smallest diffeology respectively.}.
\end{ex}

\begin{ex} \label{example definition topological space diffeology}
  Let $T$ be a topological space, and let $\mathscr{D}_T$ be the set of all continuous parametrizations of $T$.
  Then $\mathscr{D}_T$ is a diffeology on $T$.
  We denote the diffeological space $(T,\mathscr{D}_T)$ by $\conti(T)$ and this correspondence defines a functor $\conti\colon \topsp\to\dflg$, 
  where we denote the category of topological spaces and continuous maps between them by $\topsp$.
\end{ex}

\begin{defi}[\IZ, 2.8] \label{definition D-topology}
  Let $X$ be a diffeological space. The largest topology on $X$ which makes all plots continuous is called the \italic{$D$-topology} on $X$. 
  With this construction, we define the functor $D$ from the category $\dflg$ to the category of topological spaces $\topsp$.
\end{defi}

It is known that the functor $D$ is the left adjoint to $\conti$ (Shimakawa-Yoshida-Haraguchi \SH, Proposition 2.1).
For a manifold $M$ equipped with the standard diffeology (\exref{example definition manifold diffeology}), the $D$-topology of $M$ is the standard topology of the manifold $M$.

\begin{defi}[\IZ, 1.66] \label{definition generating family}
  Let $X$ be a set, and let $\mathscr{F}$ be a family of parametrizations of $X$. Then there exists the smallest diffeology $\mathscr{D}$ on $X$ which contains all parametrizations of $\mathscr{F}$ as plots. 
  We call this diffeology $\mathscr{D}$ the \italic{diffeology generated by} $\mathscr{F}$. 
  A parametrization $p\colon \plotdom{p}\to X$ is a plot of $\mathscr{D}$ if and only if it satisfies the following property:
  there is an open covering $\{U_{\alpha}\}$ of $\plotdom{p}$ such that for all $\alpha$, either ${p|}_{U_{\alpha}}$ is a constant parametrization 
  or ${p|}_{U_{\alpha}}=q\circ f$ for some $q\in\mathscr{F}$ and some smooth map $f\colon U_{\alpha} \to \plotdom{q}$ (\IZ, 1.68).

  If $\mathscr{F}$ satisfies $\bigcup_{q\in \mathscr{F}} q(\plotdom{q})=X$, we call $\mathscr{F}$ a \italic{covering generating family}. 
  In this case, parametrization $p\colon \plotdom{p}\to X$ is a plot of $\mathscr{D}$ if and only if it satisfies the following property:
  there is an open covering $\{U_{\alpha}\}$ of $\plotdom{p}$ such that for all $\alpha$, ${p|}_{U_{\alpha}}=q\circ f$ for some $q\in\mathscr{F}$ and some smooth map $f\colon U_{\alpha} \to \plotdom{q}$.
\end{defi}

\begin{ex} \label{example manifold generating family}
  Let $M$ be an $n$-dimensional manifold with an atlas $\{\varphi_{\alpha}\colon U_{\alpha}\to V_{\alpha}\}_{\alpha\in A}$, 
  where $U_{\alpha}$ is an open set of $M$, $V_{\alpha}$ is an open set of $\mathbb{R}^n$. 
  $\{\varphi_{\alpha}^{-1}\colon V_{\alpha}\to U_{\alpha}\}_{\alpha\in A}$ is a generating family of the standard diffeology on $M$, which is defined in \exref{example definition manifold diffeology}.
\end{ex}

\begin{ex} \label{example definition wire diffeology}
  Let $n\geq 2$, and let $\mathscr{F}$ be the set of all 1-plots of $\mathbb{R}^n$ with the standard diffeology. 
  The diffeology $\mathscr{D}_{wire}$ generated by $\mathscr{F}$ is called the wire diffeology and this is a different diffeology to the standard diffeology of $\mathbb{R}^n$ 
  because the standard inclusion $\mathbb{R}^2\to \mathbb{R}^n$ is a plot of standard diffeology but is not a plot of $\mathscr{D}_{wire}$.
\end{ex}

\begin{ex} \label{example definition cross_quot generating family}
  Let $\cross =\{(x,y)\in \mathbb{R}^n\mid \text{$x=0$ or $y=0$}\}$, and let $\alpha\colon \mathbb{R}\to \cross;t\mapsto (t,0),\beta\colon \mathbb{R}\to \cross;t\mapsto (0,t)$. 
  We set $\mathscr{F}=\{\alpha,\beta\}$, and we define the diffeology $\mathscr{D}_{quot}$ on $\cross$ which is generated by $\mathscr{F}$. 
  We denote the diffeological space $(\cross,\mathscr{D}_{quot})$ by $\crossquot$.
\end{ex}

Let $X$ be a set. All diffeologies on $X$ form complete lattice with respect to the standard inclusion. 
For any family $\{\mathscr{D}_{\alpha}\}_{\alpha\in A}$ of diffeologies on $X$, its infimum is $\bigcap_{\alpha\in A}\mathscr{D}_{\alpha}$ (this automatically satisfies the axiom of diffeology)
and its supremum is the diffeology generated by $\bigcup_{\alpha\in A}\mathscr{D}_{\alpha}$.

\begin{defi}[\IZ, 1.50] \label{definition quotient diffeology}
  Let $X$ be a diffeological space, and let $\equivrel$ an equivalence relation on $X$. 
  Let $Y=X/{\equivrel}$, and let $\pi\colon X\to Y$ be the quotient map. 
  We define the \italic{quotient diffeology} on $Y$ to be the smallest diffeology making the quotient map $\pi$ smooth.
  A parametrization $p\colon \plotdom{p}\to Y$ is a plot of the quotient diffeology if and only if it satisfies the following property:
  there is an open covering $\{U_{\alpha}\}$ of $U$ such that for all $\alpha$, ${p|}_{U_{\alpha}}=\pi\circ f$ for some smooth map $f\colon U_{\alpha} \to X$.

  Let $X$ and $Y$ be two diffeological spaces, and let $f\colon X\to Y$ be a surjective smooth map. 
  The map $f$ is called a \italic{subduction} if the smallest diffeology making $f$ smooth is equal to the diffeology on $Y$. 
  In particular, the quotient map $\pi\colon X\to X/{\equivrel}$ is a subduction.
\end{defi}

\begin{defi}[\IZ, 1.33] \label{definition subdiffeology}
  Let $X$ be a diffeological space, and let $Y$ be a subset of $X$. Let $i\colon Y\to X$ be the standard inclusion map. 
  We define the \italic{subdiffeology} $\mathscr{D}_Y$ on $Y$ to be the largest diffeology making $i$ smooth, that is, $\mathscr{D}_Y=\{p\colon \plotdom{p}\to Y\mid \text{$i\circ p$ is a plot of $X$}\}$.
  
  Let $X$ and $Y$ be two diffeological spaces, and let $f\colon X\to Y$ be an injective smooth map. 
  The map $f$ is called an \italic{induction} if the largest diffeology making $f$ smooth is equal to the diffeology on $X$. 
  In particular, the inclusion map $i\colon X\to Y$ where $X$ is a subset of $Y$ equipped with the subdiffeology from $Y$ is an induction.
\end{defi}

\begin{defi}[\IZ, 1.57] \label{definition functional diffeology}
  Let $X$ and $Y$ be two diffeological spaces. We define the \italic{functional diffeology} on $C^{\infty}(X,Y)$ as follows. A parametrization $p\colon \plotdom{p}\to C^{\infty}(X, Y)$ is a plot if the map
  \[
    \ad(p)\colon \plotdom{p}\times X\to Y;(u,x)\mapsto p(u)(x)
  \]
  is a smooth map from the product diffeological space $\plotdom{p}\times X$ to $Y$.
\end{defi}

With these constructions, we have:

\begin{thm}[\baez, Theorem 5.25] \label{theorem dflg are cartesian closed}
  The category $\dflg$ is complete, cocomplete and cartesian closed.
\end{thm}
\section{Internal tangent spaces of diffeological spaces} \label{section internal tangent spaces of diffeological space}

In this section, we discuss the internal tangent space on a diffeological space, which is one of the generalizations of tangent spaces on smooth manifolds.
This tangent space was defined in \CW. 

\subsection{Internal tangent spaces} \label{subsection internal}

First, we define some categories. 

\begin{defi} \label{definition local map}
  Let $X$, $Y$ be two diffeological spaces and $x\in X$, $y\in Y$. Let 
  \[
    \locmap((X,x),Y)=\{f\colon U\to Y\mid \text{$x\in U$, $U$ is open set of $X$}\},
  \]
  \[
    \locmap((X,x),(Y,y))=\{f\in \locmap((X,x),Y)\mid f(x)=y\}.
  \]
  We define an equivalence relation on $\locmap((X,x),Y)$:
  \[
    f\equivrel g\colon \Leftrightarrow \text{there exists a $D$-open set $(x\in)\; U\subset \dom(f)\cap \dom(g)$ such that ${f|}_U={g|}_U$}.
  \]
  We define
  \[
    \germ((X,x),Y)=\locmap((X,x),Y)/{\equivrel}, \;\germ((X,x),(Y,y))=\locmap((X,x),(Y,y))/{\equivrel}.
  \]
\end{defi}

\begin{defi} \label{definition euclidean category}
  We define some categories as follows.
  \begin{itemize}
    \item $\encl$ : the category of Euclidean open sets and smooth maps between them.
    $\encl$ is a full subcategory of $\dflg$.
    \item $\dflgbased$ : the category of diffeological spaces with a based point and based smooth maps between them.
    \item $\enclbased$ : the category of Euclidean open sets with a based point and based smooth maps between them.
    $\enclbased$ is a full subcategory of $\dflgbased$.
    \item $\enclzero$ : the category of Euclidean open sets which contain the origin and smooth maps that preserve the origin.
    $\enclzero$ is a full subcategory of $\enclbased$.
    \item $\dflgloc$ : the category whose objects are based diffeological spaces and the morphism set from $(X,x)$ to $(Y,y)$ is the set $\germ((X,x),(Y,y))$.
    \item $\enclloc$ : the full subcategory of $\dflgloc$ whose objects are Euclidean open sets with base point.
    \item $\enclzeroloc$ : the full subcategory of $\enclloc$ whose base points of objects are the origin.
  \end{itemize}
  We denote the
   fully faithful inclusion functors by $\incl\colon\encl\to\dflg,\inclbased\colon\enclbased\to\dflgbased,\inclzero\colon\enclzero\to\dflgbased,\inclloc\colon\enclloc\to\dflgloc,\inclzeroloc\colon\enclzeroloc\to\dflgloc$ respectively.
\end{defi}
We define the internal tangent spaces of diffeological spaces as the left Kan extension of tangent functor from the category $\enclbased$.

\begin{defi} \label{definition tangent functor from encl}
  We define the functor $T^{\enclloc}\colon \enclloc\to\vect$ as follows:
  \begin{description}
    \item [Objects] $(U,u)\mapsto T_u(U)$.
    \item [Morphisms] $([f]\colon (U,u)\to (V,v))\mapsto (df_u\colon T_u(U)\to T_v(V))$.
  \end{description}
  This functor is well-defined with respect to the representatives of morphisms because tangent spaces of manifolds are locally determined.
  Moreover, we define the functor $\encltoloc\colon\enclbased\to\enclloc$ as follows:
  \begin{description}
    \item[Objects] $(U,u)\mapsto (U,u)$. 
    \item[Morphisms] $(f\colon (U,u)\to (V,v))\mapsto ([f]\colon (U,u)\to (V,v))$.
  \end{description}
  Similarly, we define the functor $\dflgtoloc\colon\dflgbased\to\dflgloc$.
  We define the composite functor $T^{\enclbased}=T^{\enclloc}\circ \encltoloc\colon\enclbased\to\vect$, and restricting this functor, we also define the functor $T^{\enclzero}\colon\enclzero\to\vect$. 
\end{defi}

\begin{defi} \label{definition plot category}
  Let $X$ be a diffeological space, and let $x\in X$. If we regard $(X,x)$ as an object of $\dflgbased$, we think of the comma category $\inclzero\downarrow (X,x)$. 
  We denote the set of objects of $\inclzero\downarrow (X,x)$ by $\plots(X,x)$ and we call the elements of this set \italic{plots centered at $x$}.
\end{defi}

Using these notations, we define the internal tangent spaces as follows. 

\begin{defi} \label{definition internal}
  The \italic{internal tangent functor} $\internalfunctor\colon\dflgbased\to\vect$ is the left Kan extension of $T^{\enclzero}\colon\enclzero\to\vect$ along the inclusion functor $\inclzero\colon\enclzero\to\dflgloc$, that is:
  \[
    \internalfunctor(X,x) = \Lan_{\inclzero} T^{\enclzero}(X,x).
  \]
  Moreover, we construct the image of the left Kan extension concretely by using the calculation formula of the pointwise left Kan extension as follows:
  \[
    \internalfunctor(X,x) = \bigoplus_{p\in\plots(X,x)} T_0 (\plotdom{p})/R,
  \]
  where
  \[
    R=\langle v-w\mid \text{$v\in T_0 (\plotdom{p})$, $w\in T_0 (\plotdom{q})$, $T(f)v=w$, where $p=q\circ f$.}\rangle.
  \]
  We denote the vector space $\internalfunctor(X,x)$ by $\internal{x}{X}$ and call it the \italic{internal tangent space of $X$ at $x$}.
  Even if we restrict objects of $\enclzero$ only to all connected open sets, we have the same functor, and this is the same as the definition in \CW.
  The internal tangent functor is also isomorphic to the left Kan extension of the functor $T^{\enclbased}\colon\enclbased\to\vect$ along $\inclbased\colon\enclbased\to\dflgbased$ 
  because the inclusion functor $\enclzero\to\enclbased$ is a category equivalent.
\end{defi}

The correspondence of morphisms by the internal tangent functor is concrete.
Since for each plot $p\colon \plotdom{p}\to X$, there exists a natural linear map $T_0 (\plotdom{p})\to \internal{x}{X}$, 
we denote the image of $v\in T_0 (\plotdom{p})$ by this linear map by $[p,v]\in \internal{x}{X}$.
Using this notation, we describe the induced linear map concretely.
\begin{prop} \label{proposition internal tangent space morphism}
  Let $X$, $Y$ be two diffeological spaces, let $x\in X$, and let $f\colon X\to Y$ be a smooth map. 
  The induced map $T(f)\colon\internal{x}{X}\to \internal{f(x)}{Y}$ is described as follows:
  \[
    \internalfunctor(f)([p,v])=[f\circ p,v].
  \]
\end{prop}

The tangent space is a local invariant, so we have a more local definition.
\begin{prop}[\CW, the paragraph after Definition 3.1] \label{proposition local definition of internal tangent spaces}
  Let $X$ be a diffeological space, and let $x\in X$. We define the functor $\internalfunctor^{local}=\Lan_{\inclzeroloc} T^{\enclzeroloc}\colon\dflgloc\to\vect$.
  Then the composite functor $\internalfunctor^{local}\circ\dflgtoloc$ and $\internalfunctor$ are naturally isomorphic as functors $\dflgbased\to\vect$.
\end{prop}

\begin{proof}
  The germ functor $\inclzero\downarrow (X,x)\to\inclzeroloc\downarrow (X,x)$ is a final functor because 
  for each object $p$ of $\inclzeroloc\downarrow (X,x)$, it is also an object of $\inclzero\downarrow (X,x)$. 
\end{proof}

As a corollary, we describe the internal tangent spaces of $D$-open sets of diffeological spaces.
\begin{cor}[\CW, Proposition 3.6] \label{corollary subspace internal tangent}
  Let $X$ be a diffeological space, and let $x\in X$. Let $A$ be a $D$-open neighborhood of $x$ in $X$. Equip $A$ with a subdiffeology of $X$.
  Then the natural inclusion map induces an isomorphism $\internal{x}{A}\cong \internal{x}{X}$.
\end{cor}

\begin{proof}
$(A,x)$ and $(X,x)$ are isomorphic as objects of the category $\dflgloc$.
\end{proof}
\begin{prop}[\CW, Proposition 3.3] \label{proposition 1-dimension plots generate internal tangent}
  Let $X$ be a diffeological space, and let $x\in X$. $\internal{x}{X}$ is generated by elements $[p,v]$, where $p$ is a plot whose domain $\plotdom{p}$ is an open set of $\mathbb{R}$.
\end{prop}

\begin{proof}
  Let $[p,v]\in \internal{x}{X}$, where $p\in\plots(X,x)$ and $v\in T_0 \plotdom{p}$. Then there exists a curve $c\colon (-\epsilon,\epsilon)\to \plotdom{p}$ such that 
  $c(0)=0$ and $\frac{d}{dx}c(0)=v$. With this $c$, we obtain $[p\circ c,\frac{d}{dx}]=[p,v]$ and $p\circ c$ is a plot whose domain is an open set of $\mathbb{R}$.
\end{proof}

\begin{prop}[\CW, Proposition 3.7] \label{proposition product internal tangent}
  Let $X$, $Y$ be two diffeological spaces, and let $x\in X$, $y\in Y$. Then there is a natural isomorphism of vector spaces
  \[
    \internal{(x,y)}{X\times Y}\cong \internal{x}{X}\times \internal{y}{Y}.
  \]
\end{prop}

From here, let us discuss some examples of calculations of internal tangent spaces. 
We omit most of the proofs. See \CW. 

\begin{ex} \label{example manifold internal tangent}
  Let $M$ be a manifold, and let $x\in M$. Then the internal tangent space is isomorphic to the standard tangent space 
  because $(M,x)$ is isomorphic to $(U,0)$ as an object of the category $\dflgloc$, where $U$ is a local chart of $M$ at $x$, and $\internalfunctor$ is a strict extension of $\enclzeroloc$.
\end{ex}

\begin{ex}[\CW, Example 3.14] \label{example max/min internal tangent}
  Let $X$ be a set, and let $x\in X$. We denote the diffeological space $(X,\mathscr{D}_X^{max})$ by $X_{max}$ and $(X,\mathscr{D}_X^{min})$ by $X_{min}$
  ($\mathscr{D}_X^{max}$ and $\mathscr{D}_X^{min}$ were defined in \exref{example definition max/min diffeology}). Then there is an isomorphism of vector spaces:
  \[
    \internal{x}{X_{max}}\cong 0,\;\internal{x}{X_{min}}\cong 0. 
  \]
\end{ex}

\begin{ex} \label{example topological space internal tangent}
  Let $T$ be a topological space, and let $x\in T$. We think of the diffeological space $\conti(T)$ defined in \exref{example definition topological space diffeology}. Then there is an isomorphism of vector spaces:
  \[
    \internal{x}{\conti(T)}\cong 0.
  \]
\end{ex}

\begin{ex}[\CW, Example 3.17] \label{example crossquot internal tangent}
  There is an isomorphism of vector spaces, where $\crossquot$ is defined in \exref{example definition cross_quot generating family}:
  \[
    \internal{x}{\crossquot}\cong
    \begin{cases}
      \mathbb{R}^2 & (x=(0,0))\\
      \mathbb{R} & (x\neq (0,0)).
    \end{cases}
  \]
\end{ex}

\begin{ex}[\CW, Example 3.19] \label{example crosssub internal tangent}
  We give the subdiffeology $\mathscr{D}_{sub}$ from $\mathbb{R}^2$ to the set $\cross\subset\mathbb{R}^2$ in \exref{example definition cross_quot generating family}. 
  We denote the diffeological space $(\cross,\mathscr{D}_{sub})$ by $\crosssub$.
  There is an isomorphism of vector spaces:
  \[
    \internal{x}{\crosssub}\cong
    \begin{cases}
      \mathbb{R}^2 & (x=(0,0))\\
      \mathbb{R} & (x\neq (0,0)).
    \end{cases}
  \]
\end{ex}

\begin{remark}{}{} \label{remark crosssub is not diffeo to crossquot}
  This diffeological space $\crosssub$ is not diffeomorphic to $\crossquot$, which was defined in \exref{example definition cross_quot generating family}.
  In fact, the map 
  $p\colon \mathbb{R}\to \cross;t\mapsto 
  \begin{cases}
    (0,e^{1/t}) & (t\leq 0)\\
    (e^{-1/t},0) & (t\geq 0)
  \end{cases}$ is a plot of $\crosssub$, but is not a plot of $\crossquot$ 
  because $p$ does not factor through one of the maps $\alpha$ or $\beta$ (defined in \exref{example definition cross_quot generating family}) in the neighborhood of the point $(0,0)$.
\end{remark}

\begin{ex}[\CW, Example 3.24] \label{example orthogonal action quotient internal tangent}
  We equip $\mathbb{R}^n/O(n)$ with the quotient diffeology from $\mathbb{R}^n$. 
  Let $0$ be the image of $0\in\mathbb{R}^n$ by the natural projection $\mathbb{R}^n\to \mathbb{R}^n/O(n)$.
  There is an isomorphism of vector spaces:
  \[
    \internal{0}{\mathbb{R}^n/O(n)}\cong 0
  \]
\end{ex}

\begin{ex}[\CW, Example 3.25] \label{example half line internal tangent}
  We equip $[0,\infty)$ with the subdiffeology from $\mathbb{R}$. 
  There is an isomorphism of vector spaces:
  \[
    \internal{0}{[0,\infty)}\cong 0
  \]
\end{ex}

\begin{remark}{}{} \label{remark half line is not diffeo to orthogonal quotient}
  This diffeological space $[0,\infty)$ with the subdiffeology is not diffeomorphic to the spaces $\mathbb{R}^n/O(n)$ defined in \exref{example orthogonal action quotient internal tangent}.
  Indeed, the dimension\footnote{For every diffeological space and every point of $X$, the dimension of $X$ at the point is defined. See \IZ, 2.22.} 
  at the origin of $\mathbb{R}^n/O(n)$ is $n$ (\IZ, Exercise 50). However, the dimension at the origin of $\halfsub$ is $\infty$ (\IZ, Exercise 51).
\end{remark}

\begin{ex}[\CW, Example 3.23] \label{example irrational torus internal tangent}
  Let $\theta \in \mathbb{R}\backslash \mathbb{Q}$. Since $\mathbb{Z}+\theta \mathbb{Z}$ is a subgroup of $\mathbb{R}$, 
  we have the quotient group $\mathbb{R}/(\mathbb{Z}+\theta \mathbb{Z})$. The diffeological space $\mathbb{R}/(\mathbb{Z}+\theta \mathbb{Z})$ equipped with the quotient diffeology from $\mathbb{R}$ 
  is denoted by $\irrationaltorus$ and called the irrational torus of slope $\theta$.
  Then for each $x\in \irrationaltorus$, there is an isomorphism of vector spaces:
  \[
    \internal{x}{\irrationaltorus}\cong \mathbb{R}.
  \]
\end{ex}

\begin{ex} \label{example n line gluing internal tangent}
  We equip $\coprod_{\alpha\in A}\mathbb{R}_{\alpha}$ with the sum diffeology where for each $\alpha\in A$, $\mathbb{R}_{\alpha}=\mathbb{R}$ and 
  we define an equivalence relation $\equivrel$ on $\coprod_{\alpha\in A}\mathbb{R}_{\alpha}$ as follows:
  \[
    (x\in\mathbb{R}_{\alpha})\equivrel (y\in\mathbb{R}_{\beta})\Longleftrightarrow x=y\geq 0
  \]
  Let $Y_{A}=(\coprod_{\alpha\in A}\mathbb{R}_{\alpha})/{\equivrel}$ be the quotient diffeological space (when $A=\{1,\dots,n\}$, then we denote $Y_A$ by $Y_n$). There is an isomorphism of vector spaces:
  \[
    \internal{[0]}{Y_{A}}\cong\bigoplus_{\alpha\in A}\mathbb{R}.
  \]
\end{ex}

\begin{proof}
Let $i_{\alpha}\colon\mathbb{R}_{\alpha}\to Y_A$ be the natural inclusion for each $\alpha\in A$.
We construct a linear map $\gamma\colon\internal{[0]}{Y_{A}}\to\bigoplus_{\alpha\in A}\mathbb{R}$. 
Take an element $[p,v]\in\internal{[0]}{Y_{A}}$ where $p\colon \plotdom{p}\to Y_A$ is a plot of $Y_A$ and $v$ is an element of $T_0(\plotdom{p})$. 
Then there exists some $\alpha\in A$ and some smooth map $f\colon W\to \plotdom{p}$ such that ${p|}_W=i_{\alpha}\circ f$, where $W\subset \plotdom{p}$ is an open neighborhood of $0$. 
If there exist two different $\alpha$, $\beta\in A$ and two different smooth maps $f_{\alpha}\colon W_{\alpha}\to\plotdom{p}$, $f_{\beta}\colon W_{\beta}\to\plotdom{p}$ 
such that ${p|}_{W_{\alpha}\cap W_{\beta}}={i_{\alpha}\circ f_{\alpha}|}_{W_{\alpha}\cap W_{\beta}}={i_{\beta}\circ f_{\beta}|}_{W_{\alpha}\cap W_{\beta}}$, then it follows that $d(f_{\alpha})_0=d(f_{\beta})_0=0$. 
Indeed, if one of these derivatives is not equal to zero, then $p$ does not factor through only one of $i_{\alpha}$ and $i_{\beta}$. 
When there exists only one $\alpha\in A$ which satisfies the previous condition, 
we define $\gamma([p,v])\in \bigoplus_{\alpha\in A}\mathbb{R}\cong \bigoplus_{\alpha\in A} T_0(\mathbb{R}_{\alpha})$ to be an element such that 
its $\alpha$-th component is $T(f_{\alpha})v\in T_0(\mathbb{R}_{\alpha})$ and the others are $0$. 
Also, if there exist at least two $\alpha\in A$ which satisfies the previous condition, then we set that $\gamma([p,v])=0$. 
It is clear that this map $\gamma$ is well defined. 
Also, take $\mu\colon\bigoplus_{\alpha\in A}\mathbb{R}\to\internal{[0]}{Y_{A}}$ to be a linear map which sends $j_{\alpha}$ to $[i_{\alpha},\frac{d}{dx}]$, 
where $j_{\alpha}\in \bigoplus_{\alpha\in A}\mathbb{R}$ is the element such that its $\alpha$-th component is $1$ and the others are $0$. 
Since $\gamma$ and $\mu$ are inverses of each other, it means that $\internal{[0]}{Y_{A}}\cong\bigoplus_{\alpha\in A}\mathbb{R}$. 
\end{proof}

\begin{ex}[\CW, Example 3.22] \label{example wire diffeology internal tangent}
  Let $\mathbb{R}^n_{wire}$ be the diffeological space $(\mathbb{R}^n, \mathscr{D}_{wire})$ defined in \exref{example definition wire diffeology}, and let $x\in \mathbb{R}^n_{wire}$. 
  The vector space $\internal{x}{\mathbb{R}^n_{wire}}$ is uncountably-infinite dimensional if $n\geq 2$. 
\end{ex}

\begin{ex} \label{example crushing y-axis internal tangent}
  Let $\equivrel$ be the equivalence relation on $\mathbb{R}^2$ such that $s\equivrel t\Longleftrightarrow s=t$ or $s$ and $t$ are both elements of $\{(x,y)\in \mathbb{R}^2\mid x=0\}$.
  We denote the quotient diffeological space $\mathbb{R}^2/\equivrel$ by $\mathbb{R}^2/\{x=0\}$. Then the internal tangent space of $\mathbb{R}^2/\{x=0\}$ at $[(0,0)]$ is uncountably-infinite dimensional.
\end{ex}

\begin{proof}
For every $t\in\mathbb{R}$, take a smooth map $p_t\colon\mathbb{R}\to \mathbb{R}^2/\{x=0\};x\mapsto [(x,t)]$. 
We prove that all $[p_t,\frac{d}{dx}]\in\internal{[(0,0)]}{\mathbb{R}^2/\{x=0\}}$ for each $t\in \mathbb{R}$ are linearly independent.
To show this, for every $s,t\in\mathbb{R}$ such that $s>t$, define a smooth map $i_{s,t}\colon\mathbb{R}^2/\{x=0\}\to \mathbb{R};[(x,y)]\mapsto x\lambda \left(\dfrac{y-t}{s-t}\right)$, 
where $\lambda\colon\mathbb{R}\to[0,1]$ is a smooth map such that for any $x\leq 0$, $\lambda(x)=0$ and for any $x\geq 1$, $\lambda(x)=1$. 
The map $i_{s,t}$ is smooth for any $s,t\in \mathbb{R}$ because of the universality of the quotient diffeology.
Note that $i_{s,t}\circ p_r=0$ for $r\leq t$ and $i_{s,t}\circ p_r={\id}_{\mathbb{R}}$ for $r\geq s$. 
Let $[p_{t_1},\frac{d}{dx}],\dots,[p_{t_k},\frac{d}{dx}]$ be any finite elements of $\{[p_t,\frac{d}{dx}]\}_{t\in \mathbb{R}}$. 
We rearrange $[p_{t_1},\frac{d}{dx}],\dots,[p_{t_k},\frac{d}{dx}]$ to satisfy that $t_1<\dots<t_k$.
We set that 
\[
  c_1\left[p_{t_1},\frac{d}{dx}\right]+\dots+c_k\left[p_{t_k},\frac{d}{dx}\right]=0
\]
for $c_1,\dots,c_k\in \mathbb{R}$. Applying $\internalfunctor(i_{t_k,t_{k-1}})$ to both sides of this equality, we have $c_k=0$. 
Continuing this process, we finally obtain $c_1=\dots =c_k=0$. This means that $\{[p_t,\frac{d}{dx}]\}_{t\in \mathbb{R}}$ are linearly independent.
Namely, $\internal{[(0,0)]}{\mathbb{R}^2/\{x=0\}}$ is uncountably-infinite dimensional.
\end{proof}

\subsection{Internal tangent bundles} \label{subsection diffeology on internal bundle}
We introduce diffeologies on the internal tangent space. For this purpose, we first briefly review the definitions and properties of concepts related to diffeological vector spaces.
All definitions and properties in this section are found in 
Christensen-Wu \DVS, \IZ, Wu \wu, and \CW.

\begin{defi}[\DVS, Definition 2.3, \wu, Definition 2.1] \label{definition diffeological vector space}
  A \italic{diffeological vector space} is a vector space $V$ together with a diffeology, such that the addition map $V\times V\to V$ 
  and the scalar multiplication map $\mathbb{R}\times V\to V$ are both smooth, where we equip $V\times V$ and $\mathbb{R}\times V$ with the product diffeology. 
\end{defi}

\begin{defi} \label{definition smooth linear map}
  A \italic{smooth linear map} between two diffeological vector spaces is a map which is smooth and linear. 
\end{defi}

For two diffeological vector spaces $V$ and $W$, we denote the set of all smooth linear maps from $V$ to $W$ by $\smoothlinear{V}{W}$, 
and the set of all linear maps from $V$ to $W$ by $\linear{V}{W}$. $\smoothlinear{V}{W}$ with the subdiffeology from $\smoothmap{V}{W}$ is a diffeological vector space. 
Therefore, we always equip $\smoothlinear{V}{W}$ with this diffeology. The evaluation map $V\times\smoothlinear{V}{W}\to W$ is smooth.

Diffeological vector spaces and smooth linear maps between them form a category denoted by $\dvect$. We have the forgetful functors $\dvect\to \dflg$ and $\dvect\to\vect$. 
The category $\dvect$ of diffeological vector spaces is both complete and cocomplete (\wu, Theorem 3.1 and Theorem 3.3). 
Moreover, the forgetful functor $\dvect\to\dflg$ has the left adjoint (\wu, Proposition 3.5), 
and the forgetful functor from $\dvect\to\vect$ has both the left adjoint and the right adjoint.
For any diffeological vector space $V$, the diffeology on the image of $V$ by this right adjoint is called the \italic{fine diffeology}, which is the smallest diffeology making $V$ into a diffeological vector space. 

\begin{defi}[\DVS, Definition 3.1, \IZ, 3.7 and \wu, Definition 5.2] \label{definition fine diffeological vector space}
  A vector space with the fine diffeology is called a \italic{fine diffeological vector space}.
\end{defi}

It is known that every linear subspace and every quotient vector space of a fine diffeological vector space is again a fine diffeological vector space. 
Also, the finite product of fine diffeological vector spaces is again a fine diffeological vector space. 

In order to consider the tangent bundles on diffeological spaces, we introduce some concepts related to bundles whose fibers are diffeological vector spaces. 

\begin{defi} \label{definition vector space with diffeology over X}
  Let $X$ be a diffeological space. \italic{A vector space over $X$} is a diffeological space $V$, a smooth map $p\colon V\to X$ and a vector space structure on each of the fibers $V_x=p^{-1}(x)$. 
  Let $p\colon V\to X$ be vector space over $X$. We call $V$ the \italic{total space} and $p$ the \italic{projection}. 
  The category $\vsdcat$ has as objects the vector spaces with diffeology over diffeological spaces, and as morphisms the commutative squares
  \[
    \xymatrix{
    V \ar[r]^g \ar[d] & W \ar[d]\\
    X \ar[r]^f & Y
    }
  \] 
  in $\dflg$ such that for each $x\in X$, ${g|}_{V_x}\colon V_x\to W_{f(x)}$ is linear.
\end{defi}

\begin{defi}[\CW, Definition 4.5] \label{definition diffeological vector space over X}
  Let $X$ be a diffeological space. A \italic{diffeological vector space over $X$} is a diffeological space $V$, a smooth map $p\colon V\to X$ and a vector space structure on 
  each of the fibers $p^{-1}(x)$ such that the addition map $V\times_X V\to V$ and the scalar multiplication map $\mathbb{R}\times V\to V$ and the zero section $X\to V$ are smooth. 
  Here $\mathbb{R}\times V$ has the product diffeology and $V\times_X V$ has the pullback diffeology, i.e., the subdiffeology from the product space $V\times V$. 
  If $X$ is a point, this coincides with the concept of diffeological vector spaces. 
  We denote by $\dvscat$ the full subcategory of $\vsdcat$ whose objects are diffeological vector spaces over diffeological spaces.
\end{defi}

\begin{prop}[\CW, Proposition 4.16] \label{proposition vsd is complete and cocomplete}
  The category $\vsdcat$ is complete and cocomplete.
\end{prop}

\begin{ex} \label{example definition tangent bundle on Euclidean open set}
  Let $U$ be a Euclidean open set. We have the tangent bundle $\pi_U\colon TU \to U$ in the standard meaning, where $TU=\coprod_{x\in U}T_x U$. 
  We equip $TU$ with the standard manifold diffeology. With this diffeology, $\pi_U$ is a vector space over $U$. This is also a diffeological vector space over $X$.
  Let $f\colon U\to V$ be a smooth map between two Euclidean open sets. Then we have the induced map $Tf\colon TU\to TV$ between tangent bundle. 
  The following commutative diagram is a morphism in $\dvscat$ and $\vsdcat$:
  \[
    \xymatrix{
    TU \ar[r]^{Tf} \ar[d] & TV \ar[d]\\
    U \ar[r]^f & V
    }
  \]
  By this correspondence, we have the tangent bundle functors $T^{dvs}\colon \encl\to\dvscat$, $T^{vsd}\colon \encl\to\vsdcat$.
  We can also do this construction by using the category $\mfd$ instead of $\encl$.
\end{ex}

With these in mind, we are now ready to define diffeologies on the internal tangent bundle.

\begin{defi}[\CW, Definition 4.1, Definition 4.10] \label{definition internal tangent bundle}
  We define $\hectorbundle\colon\dflg\to\vsdcat$ as the left Kan extension of $T^{vsd}\colon\encl\to\vsdcat$ along the inclusion functor $i\colon \encl\to\dflg$ (\exref{example definition tangent bundle on Euclidean open set}), i.e., 
  $\hectorbundle=\Lan_{i} T^{vsd}$. Also, we define $\dvsbundle\colon\dflg\to\dvscat$ as the left Kan extension of $T^{dvs}\colon\encl\to\dvscat$ along the inclusion functor $i\colon \encl\to\dflg$, i.e., 
  $\dvsbundle=\Lan_{i} T^{dvs}$.

  Let $X$ be a diffeological space. We often denote the projection of $\hectorbundle (X)$ by $\pi^{vsd}_X$ and the total space by $\hectorbundle (X)$. 
  Using this notation, $\pi^{vsd}_X\colon\hectorbundle(X)\to X$ is an object of the category $\vsdcat$.
  Also, we often denote the projection of $\dvsbundle (X)$ by $\pi^{dvs}_X$ and the total space by $\dvsbundle (X)$. 
\end{defi}
These definitions are equivalent to the definitions in \CW, Definition 4.1 and 4.10, as proved in \CW, Theorem 4.17. 
Let $X$ be a diffeological space, and let $x\in X$. We often equip the fiber of $x$ with the subdiffeology from $\hectorbundle(X)$ and $\dvsbundle(X)$ respectively.
Since the fiber of $x$ is bijective to the set $\internal{x}{X}$, we often write these diffeological vector spaces as $\hectorinternal{x}{X}$ and $\dvsinternal{x}{X}$ respectively.
We consider $\hectorbundle(X)$ as a diffeological space whose base set is the set $\coprod_{x\in X} \internal{x}{X}$.
In this case, the diffeology on $\hectorbundle(X)$ is generated by $i_U\circ p\circ f\colon \plotdom{f}\to \coprod_{x\in X} \internal{x}{X}$, 
where $p\colon \plotdom{p}\to TU$ is a plot of $TU$, $i_U\colon TU\to\coprod_{x\in X} \internal{x}{X}$ is the natural morphism, and $f\colon \plotdom{f}\to\plotdom{p}$ is a smooth map between two Euclidean open sets. 
Also, the diffeology on $\dvsbundle(X)$ is the smallest diffeology which contains the diffeology of $\hectorbundle(X)$ and makes $\coprod_{x\in X} \internal{x}{X}\to X$ into a diffeological vector space over $X$. 

\begin{prop}[\CW, Proposition 4.22] \label{proposition internal tangent of fine diffeological vector space}
  Let $V$ be a fine diffeological vector space. Then $\internal{0}{V}\cong V$ as diffeological vector spaces.
\end{prop}

Let $X$ be a diffeological space, and let $x\in X$. Let $G$ be a group which acts on $X$. 
In this section, we would like to consider an action such that for every $g\in G$, the induced map $g\cdot \colon X\to X$ is smooth.
We call such an action a \italic{pointwise smooth action}.
In this situation, we have induced linear action of $G_x$ on $\internal{x}{X}$, where $G_x=\{g\in G\mid g\cdot x=x\}$ is the stabilizer of $x$.

\begin{prop} \label{proposition definition linear action on internal tangent}
  Let $X$ be a diffeological space, let $x\in X$, and let $G$ be a group. 
  If $\varphi$ is a pointwise smooth action of $G$ on $X$, 
  then $\varphi$ induces a smooth linear action of $G_x$ on $\internal{x}{X}$.
\end{prop}

\begin{proof}
For each $g\in G$, we denote the induced diffeomorphism by $\varphi_g\colon X\to X;y\mapsto g\cdot y$.
Let $g\in G_x$. Since $\varphi_g(x)=x$ holds, we have the induced linear isomorphism $\internalfunctor(\varphi_g)\colon\internal{x}{X}\to \internal{x}{X}$.
For $g\in G_x$, we write $\tilde{\varphi}_g=\internalfunctor(\varphi_g)$. The correspondence $G_x\to \gl{\internal{x}{X}};g\mapsto \tilde{\varphi}_g$ defines a linear action of $G_x$ on $\internal{x}{X}$.
\end{proof}

Let $\varphi$ be a pointwise smooth action of a group $G$ on a diffeological space $X$, and let $x\in X$. 
In what follows, we denote the induced linear action of $G_x$ on $\internal{x}{X}$ by $\tilde{\varphi}$. 
Also, for $g\in G_x$, we denote the induced linear isomorphism $\internal{x}{X}\to\internal{x}{X}$ by $\tilde{\varphi}_g$.

\begin{prop} \label{proposition induced linear map from internal tangent of quotient}
  Let $X$ be a diffeological space, let $x\in X$, let $G$ be a group, and let $\varphi$ be a pointwise smooth action of $G$ on $X$. Let $\pi\colon X\to X/G_x$ be the standard projection.
  The map $\internalfunctor(\pi)\colon\internal{x}{X}\to \internal{[x]}{X/G_x}$ induces a smooth linear map $(\internal{x}{X})_{G_x}\to\internal{[x]}{X/G_x}$, 
  where for a diffeological vector space $V$ and a smooth linear action of $H$ on $V$, we write $V_H=V/\langle v-h\cdot v\mid v\in V, h\in H\rangle$ for the $H$-coinvariant.
\end{prop}

\begin{proof}
Let $g\in G_x$, and let $[p,v]\in \internal{x}{X}$, where $p\colon \plotdom{p}\to X$ is a plot and $v\in T_0 \plotdom{p}$ is a tangent vector. 
We have:
\[
  \internalfunctor(\pi)(\tilde{\varphi}_g[p,v])=\internalfunctor(\pi)([\varphi_g\circ p,v])=[\pi\circ\varphi_g\circ p,v]=[\pi\circ p,v]=\internalfunctor(\pi)([p,v]).
\]
Therefore, the map $\internalfunctor(\pi)$ induces a smooth linear map $(\internal{x}{X})_{G_x}\to\internal{[x]}{X/G_x}$ because of the universality of the quotient diffeological vector space.
\end{proof}
If we further assume that $X$ is a fine diffeological vector space and that the action is linear in the setting of \proref{proposition induced linear map from internal tangent of quotient},
the induced action of $G_x$ on $\internal{x}{X}$ is isomorphic to the restricted action of $G_x$ on $X$ through the isomorphism $\internal{x}{X}\cong X$ (\proref{proposition internal tangent of fine diffeological vector space}).
Moreover, the induced map in \proref{proposition induced linear map from internal tangent of quotient} is a linear diffeomorphism. We prove this in the next corollary. 

\begin{cor} \label{corollary internal tangent of quotient of fine vector space}
  Let $V$ be a fine diffeological vector space, let $G$ be a group, and let $\varphi$ be a linear action of $G$ on $V$. 
  Then the induced map $\internalfunctor(\pi)\colon(\internal{0}{V})_G\to\internal{[0]}{V/G}$ in \proref{proposition induced linear map from internal tangent of quotient} is an isomorphism 
  (automatically, this action $\varphi$ is a smooth action and $0$ is a fixed point of $\varphi$ because this action is a linear action and $V$ is fine). 
  In particular, there is an isomorphism of diffeological vector spaces:
  \[
    \internal{[0]}{V/G}\cong V_G.
  \]
\end{cor}

\begin{proof}
Since $V$ is a fine diffeological vector space, $\internal{0}{V}\cong V$ as diffeological vector spaces (\proref{proposition internal tangent of fine diffeological vector space}). 
Through this isomorphism, $\varphi$ is isomorphic to the induced linear action $\tilde{\varphi}$ on $\internal{0}{V}$. 
Therefore, we have the induced smooth linear map $V_G\to\internal{[0]}{V/G}$. We construct the inverse. 
By the universality of the quotient diffeology, 
we have the natural smooth map $\gamma\colon V/G\to V_G$, where $V/G$ is a quotient diffeological space which may not necessarily be a vector space.
It induce a smooth linear map $\internalfunctor(\gamma)\colon\internal{[0]}{V/G}\to\internal{[0]}{V_G}$. 
Since $V_G$ is a quotient diffeological vector space of a fine diffeological vector space $V$, $V_G$ is also a fine diffeological vector space. 
Therefore, $\internal{[0]}{V_G}\cong V_G$ as diffeological vector spaces (\proref{proposition internal tangent of fine diffeological vector space}) 
and this gives us a natural smooth linear map $\internalfunctor(\gamma)\colon\internal{[0]}{V/G}\to V_G$.
It is straightforward to verify that this map is the inverse of the map $V_G\to\internal{[0]}{V/G}$. 
\end{proof}
  We want a stronger sufficient condition that the map in \proref{proposition induced linear map from internal tangent of quotient} is an isomorphism. However, we do not know such a useful condition. 

\begin{ex} \label{example internal tangent horizontal quotient}
  Consider the pointwise smooth action of $\mathbb{R}$ on $\mathbb{R}^2$ given by $t\cdot (x,y)=(x+ty,y)$ for each $t\in \mathbb{R}$. 
  We write $X$ for the quotient diffeological space of $\mathbb{R}^2$ by this action. 
  There is an isomorphism of diffeological vector spaces:
  \[
    \internal{[(0,0)]}{X}\cong \mathbb{R}.
  \]
\end{ex}

\begin{proof}
Since this is a linear action on the fine diffeological vector space $\mathbb{R}^2$, we have $\internal{[(0,0)]}{X}\cong {{\mathbb{R}}^2}_{\mathbb{R}}$ (\corref{corollary internal tangent of quotient of fine vector space}). 
Two elements $(1,1)$ and $(0,1)\in \mathbb{R}^2$ are the same elements in ${{\mathbb{R}}^2}_{\mathbb{R}}$ because $1\cdot (0,1)=(1,1)$. 
Therefore, it follows that ${{\mathbb{R}}^2}_{\mathbb{R}}$ is less or equal to $1$-dimensional. 
We set $\alpha\colon\mathbb{R}\to X;t\mapsto [(0,t)]$ and $\beta\colon X\to \mathbb{R};[(a,b)]\mapsto b$. Both of these are smooth maps because of the universality of the quotient diffeology. 
Since $\beta\circ \alpha={\id}_{\mathbb{R}}$ holds, we have $\internalfunctor(\beta)\circ \internalfunctor(\alpha)={\id}_{\internal{0}{\mathbb{R}}}$. 
From this equality, we have at least one nontrivial element of $\internal{[(0,0)]}{X}$. This completes the proof. 
\end{proof}

\begin{ex} \label{example internal tangent eigenvalue>,<1}
  Consider the pointwise smooth action of $\mathbb{R}$ on $\mathbb{R}^2$ given by $t\cdot (x,y)=(2^tx, 2^{-t}y)$ for each $t\in \mathbb{R}$. 
  We write $X$ for the quotient diffeological space of $\mathbb{R}^2$ by this action. 
  There is an isomorphism of diffeological vector spaces:
  \[
    \internal{[(0,0)]}{X}\cong 0.
  \]
\end{ex}

\begin{proof}
  Since this is a linear action on the fine diffeological vector space $\mathbb{R}^2$, we have $\internal{[(0,0)]}{X}\cong {{\mathbb{R}}^2}_{\mathbb{R}}$ (\corref{corollary internal tangent of quotient of fine vector space}). 
  Two elements $(1,0)$ and $(2,0)\in \mathbb{R}^2$ are the same elements in ${{\mathbb{R}}^2}_{\mathbb{R}}$ because $2\cdot (1,0)=(2,0)$. 
  Also, two elements $(0,1)$ and $(0,2)\in \mathbb{R}^2$ are the same elements in ${{\mathbb{R}}^2}_{\mathbb{R}}$ because $2\cdot (0,2)=(0,1)$. 
  Therefore, we have ${{\mathbb{R}}^2}_{\mathbb{R}}\cong 0$. This completes the proof. 
\end{proof}
\section{Right tangent spaces and external tangent spaces of diffeological spaces} \label{section external}

\subsection{Right tangent spaces and external tangent spaces}
There are several ways, which are not necessarily equivalent to each other,
to define the tangent space of a diffeological space that extend those for manifolds. 
Here we recall the one called the external tangent space. 

\begin{prop}[\CW, the paragraph before Definition 3.10] \label{proposition germ is diffeological algebra}
  Let $X$ be a diffeological space, and let $x\in X$. We denote the set $\germ(X,x)=\germ((X,x),\mathbb{R})=\locmap((X,x),\mathbb{R})/{\equivrel}$, defined in \defref{definition local map}. 
  Since $\locmap((X,x),\mathbb{R})=\coprod_{B\in \mathscr{O}_x} C^{\infty}(B,\mathbb{R})$, where $\mathscr{O}_x=\{B\subset X\mid \text{$B$ is a $D$-open neighborhood of $x$}\}$, 
  we equip $\locmap((X,x),\mathbb{R})$ with the sum diffeology of the functional diffeology and $\germ(X,x)$ with the quotient diffeology. 
  $\germ(X,x)$ is a diffeological $\mathbb{R}$-algebra with this diffeology and the natural $\mathbb{R}$-algebra structure.
\end{prop}

We denote an element of $\germ(X,x)$ which is represented by $g\colon \coplotdom{g}\to\mathbb{R}$ by $[g, \coplotdom{g}]$, where $\coplotdom{g}$ is a $D$-open set of $X$.
The sum and product of two elements $[g, \coplotdom{g}], [h, \coplotdom{h}]\in\germ(X,x)$ are $[g+h, \coplotdom{g}\cap\coplotdom{h}]$ and $[gh, \coplotdom{g}\cap\coplotdom{h}]$ respectively.

\begin{defi} \label{definition external}
  Let $X$ be a diffeological space, and let $x\in X$. We call a map $D\colon\germ(X,x)\to\mathbb{R}$ a \italic{right tangent vector} if the following two conditions hold:
  \begin{description}
    \item[Linearity] $D$ is a linear map.
    \item[Leibniz rule] For all $G=[g, \coplotdom{g}], H=[h, \coplotdom{h}]\in \germ(X,x)$, $D(GH)=g(x)D(H)+h(x)D(G)$.
  \end{description}
  We denote the set of all right tangent vectors by $\righttangent{x}{X}$ and we give the natural vector space structure to $\righttangent{x}{X}$. 
  This vector space is called the \italic{right tangent space of $X$ at $x$}.
  Moreover, we call a right tangent vector $D\in\righttangent{x}{X}$ an \italic{external tangent vector} if the following condition also holds:
  \begin{description}
    \item[Smoothness] $D$ is a smooth map.
  \end{description}
  The set $\external{x}{X}$ of all external tangent vectors define the subspace of $\righttangent{x}{X}$, and it is called the \italic{external tangent space of $X$ at $x$}. 
  The external tangent space is the same vector space defined in \CW, Definition 3.10.
\end{defi}

It is well known that $\righttangent{x}{X}$ and $\external{x}{X}$ are both isomorphic to the standard tangent space $T_x X$ if $X$ is an open set of $\mathbb{R}^n$ or more generally, $X$ is a manifold.

We define the induced linear map between right or external tangent spaces. 
Let $f\colon X\to Y$ be a smooth map. We have the induced linear map $\rightfunctor(f)\colon\righttangent{x}{X}\to\righttangent{f(x)}{Y}$ defined as follows:
\[
  \text{For all $G=[g, \coplotdom{g}]\in\germ(Y,f(x))$ and $D\in\righttangent{x}{X}$, $(\rightfunctor(f)(D))(G)=D([g\circ f, f^{-1}(\coplotdom{g})])$}.
\]
We define the right tangent functor $\rightfunctor\colon\dflgbased\to\vect$ by using this correspondence of morphisms. 
Restricting this to external tangent vectors, we define the external tangent functor $\externalfunctor\colon\dflgbased\to\vect$. 
In fact, we have the following proposition.

\begin{prop}[\CW, the paragraph after Definition 3.10] \label{proposition external tangent forms functor}
  Let $X$ and $Y$ be two diffeological spaces, and let $x\in X$. If $D$ is an element of $\external{x}{X}$, then $\rightfunctor(f)(D)$ is an element of $\external{f(x)}{Y}$.
  In particular, we have the \italic{external tangent functor} $\externalfunctor\colon\dflg\to\vect$. 
\end{prop}

\begin{proof}
Let $D\in\external{x}{X}$. Take $D'=\rightfunctor(f)(D)$. We need to prove that $D'\colon\germ(Y,f(x))\to \mathbb{R}$ is smooth. 
Since $\germ(Y,f(x))$ has the quotient diffeology from $\locmap((Y,f(x)),\mathbb{R})$, 
it is sufficient to prove that the composition map $\locmap((Y,f(x)),\mathbb{R})\to\mathbb{R};g\mapsto D'([g,\coplotdom{g}])$ is smooth 
(by the universality of the quotient diffeology).
Moreover, since $\locmap((X,x),\mathbb{R})=\coprod_{B\in \mathscr{O}_x} C^{\infty}(B,\mathbb{R})$ holds (where $\mathscr{O}_x=\{B\subset X\mid \text{$B$ is a $D$-open neighborhood of $x$}\}$),
it is sufficient to prove that for each $B\in\mathscr{O}_x$, the composition map $C^{\infty}(B,\mathbb{R})\to\mathbb{R};g\mapsto D'([g,\coplotdom{g}])$ is smooth 
(by the universality of the sum diffeology). For simplicity, we denote this map by ${D'}_B$. 
Since this map ${D'}_B$ is the composition of $\precomp{f}\colon\smoothmap{B}{\mathbb{R}}\to\smoothmap{f^{-1}(B)}{\mathbb{R}};g\mapsto g\circ f$ and $D_{f^{-1}(B)}\colon\smoothmap{f^{-1}(B)}{\mathbb{R}}\to \mathbb{R}$, the map ${D'}_B$ is smooth.
Indeed, the map $D_{f^{-1}(B)}$ is smooth because $D$ is in $\external{x}{X}$. 
\end{proof}

If necessary, we also regard these functors as $\rightfunctor, \externalfunctor\colon\dflgloc\to\vect$. 
In fact, we can understand the right tangent space as the image of the right Kan extension of the functor $T^{\enclloc}\colon\enclloc\to\vect$ along the inclusion functor $\inclloc\colon\enclloc\to\dflgloc$ as in the next theorem.
Similarly to elements of $\germ(X,x)$, if an element $G\in \obj((X,x)\downarrow \inclloc)$ is represented by a map $g\colon \coplotdom{g}\to \coplotim{g}$ 
where $\coplotdom{g}$ is a $D$-open set of $X$ and $\coplotim{g}$ is a Euclidean open set, we write $G=[g,\coplotdom{g}]$.
For $G=[g,\coplotdom{g}]\in\obj((X,x)\downarrow \inclloc)$, we sometimes denote the codomain of $g$ by $\coplotim{G}$ and the point $g(x)$ by $G(x)$ 
because the codomain $\coplotim{g}$ and the point $g(x)$ is independent to the representative $g$ of $G$. 

\begin{thm} \label{theorem right tangent space is right Kan extension}
  Let $X$ be a diffeological space, and let $x\in X$. Then there is an isomorphism of vector spaces:
  \[
    \Ran_{\inclloc} T^{\enclloc} (X,x)\cong \righttangent{x}{X}.
  \]
\end{thm}

\begin{proof}
  It should be noted that an element of $\Ran_{\inclloc} T^{\enclloc} (X,x)$ is written 
  as an element of $\{v_G\}_{G\in\obj(\coplotcat)}\in \prod_{G\in\obj(\coplotcat)}T_{G(x)}\coplotim{G}$ which satisfies the following condition:
  \[
    \text{For all $G$, $H\in\obj(\coplotcat)$ and a morphism $s\colon G\to H$ in $\coplotcat$, $T^{\enclloc}(s)v_G=v_H$}.
  \]
  Therefore, we regard an element of $v\in\righttangent{x}{X}$ as an element of $\prod_{G\in\obj(\coplotcat)}T_{G(x)}\coplotim{G}$, 
  and we denote the $G$-th component of $v$ by $v_G$.
  Let $\alpha\colon\righttangent{x}{X}\to\Ran_{\inclloc} T^{\enclloc} (X,x)$ be a linear map defined by the following:
  \begin{equation}
    \label{compatibility}
    \text{For all $G=[g, \coplotdom{g}]\in\obj(\coplotcat)$ and $D\in\righttangent{x}{X}$, $\alpha(D)_G=\rightfunctor(g)(D)$,}
  \end{equation}
  where we regard $\rightfunctor(g)(D)\in T_{G(x)} \coplotim{G}$ because $\righttangent{G(x)}{\coplotim{G}}\cong T_{G(x)} (\coplotim{G})$ as we mentioned before. 
  Next, we define a linear map $\beta\colon\Ran_{\inclloc} T^{\enclloc} (X,x)\to\righttangent{x}{X}$. 
  To do this, let $v\in \Ran_{\inclloc} T^{\enclloc} (X,x)$, and let $G=[g, \coplotdom{g}]\in\germ(X,x)$. We define $\beta(v)$ as follows:
  \[
    \beta(v)([g,\coplotdom{g}])=v_G([{\id}_{\mathbb{R}}, \mathbb{R}]),
  \]
  where we regard $v_G$ as an element of $\righttangent{G(x)}{\coplotim{G}}$ because of the isomorphism $T_{G(x)}(\coplotim{G})\cong\righttangent{G(x)}{\coplotim{G}}$. 
  We need to prove that $\beta(v)$ as defined before is truly a right tangent vector. 
  First, let $G=[g, \coplotdom{g}], H=[h, \coplotdom{h}]\in \germ(X,x)$. 
  We define $(g,h)\colon\coplotdom{g}\cap\coplotdom{h}\to\mathbb{R}^2;y\mapsto (g(y),h(y))$ and $(G,H)=[(g,h),\coplotdom{g}\cap\coplotdom{h}]$.
  Since $v$ satisfies the condition \ref{compatibility}, the following equalities hold:
  \[
    v_{G+H}=T(\plus)v_{(G,H)}, v_G=T(\prl)v_{(G,H)}, v_H=T(\prr)v_{(G,H)},
  \]
  where we define $\plus\colon\mathbb{R}^2\to\mathbb{R};(a,b)\mapsto a+b$, $\prl\colon\mathbb{R}^2\to\mathbb{R};(a,b)\mapsto a$, $\prr\colon\mathbb{R}^2\to\mathbb{R};(a,b)\mapsto b$.
  Therefore, the following equality holds:
  \begin{equation*}
    \begin{split}
    &v_{G+H}([{\id}_{\mathbb{R}},\mathbb{R}])=T(\plus)v_{(G,H)}([{\id}_{\mathbb{R}},\mathbb{R}])=v_{(G,H)}([{\id}_{\mathbb{R}}\circ\plus,{\plus}^{-1}(\mathbb{R})])\\
    &=v_{(G,H)}([\plus,\mathbb{R}^2])=v_{(G,H)}([\prl,\mathbb{R}^2]+[\prr,\mathbb{R}^2])\\
    &=v_{(G,H)}([\prl,\mathbb{R}^2])+v_{(G,H)}([\prr,\mathbb{R}^2])\\
    &=T(\prl)v_{(G,H)}([{\id}_{\mathbb{R}},\mathbb{R}])+T(\prr)v_{(G,H)}([{\id}_{\mathbb{R}},\mathbb{R}])\\
    &=v_G([{\id}_{\mathbb{R}},\mathbb{R}])+v_H([{\id}_{\mathbb{R}},\mathbb{R}]).
    \end{split}
  \end{equation*}
  The fifth equality holds because $v_{(G,H)}\in \righttangent{(G(x),H(x))}{\mathbb{R}^2}$ is a linear map $\germ(\mathbb{R}^2,(G(x),H(x)))\to\mathbb{R}$.
  Similarly, let $G=[g, \coplotdom{g}]\in \germ(X,x)$, and let $t\in\mathbb{R}$. We define $t\colon\mathbb{R}\to\mathbb{R};a\mapsto ta$. 
  The following equality holds:
  \[
    tv_G([{\id}_{\mathbb{R}},\mathbb{R}])=v_G([t\cdot {\id}_{\mathbb{R}},\mathbb{R}])=T(t)v_G([{\id}_{\mathbb{R}},\mathbb{R}])=v_{tG}([{\id}_{\mathbb{R}},\mathbb{R}]).
  \]
  Therefore, we have proved that $\beta(v)$ is a linear map. Second, we prove that $\beta(v)$ satisfies the Leibniz rule. 
  Let $G=[g, \coplotdom{g}], H=[h, \coplotdom{h}]\in \germ(X,x)$. We define $\kakeru\colon\mathbb{R}^2\to\mathbb{R};(a,b)\mapsto ab$. 
  As in the proof of linearity, the following equality holds:
  \begin{equation*}
    \begin{split}
    &v_{GH}([{\id}_{\mathbb{R}},\mathbb{R}])=T(\kakeru)v_{(G,H)}([{\id}_{\mathbb{R}},\mathbb{R}])=v_{(G,H)}([\kakeru,\mathbb{R}^2])=v_{(G,H)}([\prl,\mathbb{R}^2]*[\prr,\mathbb{R}^2])\\
    &=\prl((G(x),H(x)))v_{(G,H)}([\prr,\mathbb{R}^2])+\prr((G(x),H(x)))v_{(G,H)}([\prl,\mathbb{R}^2])\\
    &=G(x)v_H([{\id}_{\mathbb{R}},\mathbb{R}])+H(x)v_G([{\id}_{\mathbb{R}},\mathbb{R}]).
    \end{split}
  \end{equation*}
  The fourth equality holds because $v_{(G,H)}\in \righttangent{(G(x),H(x))}{\mathbb{R}^2}$ satisfies the Leibniz rule. 
  Since we have proved that $\beta(v)$ is a right tangent vector, $\beta$ is an well-defined map. Moreover, it is clear that $\beta$ is a linear map. 
  
  We prove that $\beta\circ\alpha={\id}_{\righttangent{x}{X}}$ and $\alpha\circ\beta={\id}_{\Ran_{\inclloc} T^{\enclloc} (X,x)}$. 
  First, let $D\in\righttangent{x}{X}$, and let $[g,\coplotdom{g}]\in\germ(X,x)$. The following equality holds:
  \begin{equation*}
    \begin{split}
    &\beta\circ\alpha(D)([g,\coplotdom{g}])=\alpha(D)_G([{\id}_{\mathbb{R}},\mathbb{R}])=\rightfunctor(g)(D)([{\id}_{\mathbb{R}},\mathbb{R}])\\
    &=D([{\id}_{\mathbb{R}}\circ g, g^{-1}(\mathbb{R})])=D([g,\coplotdom{g}]).
    \end{split}
  \end{equation*}
  Therefore, $\beta\circ\alpha={\id}_{\righttangent{x}{X}}$ holds. Second, let $v\in\righttangent{x}{X}$, and let $[g,\coplotdom{g}]\in\obj(\coplotcat)$. 
  Since $\alpha\circ \beta(v)_{[g,\coplotdom{g}]}=\rightfunctor(g)(\beta(v))$ holds, it is enough to show that $\rightfunctor(g)(\beta(v))=v_{[g,\coplotdom{g}]}$. 
  For any $[h,\coplotdom{h}]\in\germ(X,x)$, the following equality holds:
  \begin{equation*}
    \begin{split}
    &(\rightfunctor(g)(\beta(v)))([h,\coplotdom{h}])=\beta(v)([h\circ g, g^{-1}(\coplotdom{h})])\\
    &=v_{[h\circ g, g^{-1}(\coplotdom{h})]}([{\id}_{\mathbb{R}}, \mathbb{R}])=T(h)v_{[g, g^{-1}(\coplotdom{h})]}([{\id}_{\mathbb{R}}, \mathbb{R}])\\
    &=\rightfunctor(h)v_{[g, \coplotdom{g}]}([{\id}_{\mathbb{R}}, \mathbb{R}])=v_{[g, \coplotdom{g}]}([{\id}_{\mathbb{R}}\circ h, h^{-1}(\mathbb{R})])=v_{[g, \coplotdom{g}]}([h, \coplotdom{h}]).
    \end{split}
  \end{equation*}
  Therefore, we have proved that $\rightfunctor(g)(\beta(v))=v_{[g,\coplotdom{g}]}$ and this means that $\alpha\circ\beta(v)=v$. 

  Consequently, we have the isomorphism $\Ran_{\inclloc} T^{\enclloc} (X,x)\cong \righttangent{x}{X}$.
\end{proof}

Thanks to this theorem, we can calculate the right tangent space as the image of the right Kan extension of the standard tangent functor $T^{\enclloc}\colon\enclloc\to\vect$ 
along the inclusion functor $\inclloc\colon\enclloc\to\dflg$. Since $\enclloc$ and $\enclzeroloc$ are category equivalent, we can also calculate the right tangent functor as 
the right Kan extension of $T^{\enclzeroloc}\colon\enclzeroloc\to\vect$ along the inclusion functor $\inclzeroloc\colon\enclzeroloc\to\dflg$. 

\begin{cor} \label{corollary calculation of right tangent}
  Let $X$ be a diffeological space, and let $x\in X$. Let $f\colon X\to \coplotim{f}$ be a smooth map which satisfies the following conditions:
  \begin{enumerate}
    \item For any smooth map $g\colon \coplotdom{g}\to\coplotim{g}$ where $\coplotdom{g}$ is a $D$-open set,
          there exists at least one morphism $h_g\colon [f,X]\to [g,\coplotdom{g}]$ in $\coplotcat$. 
    \item If $h_g, h_g'\colon[f,X]\to [g,\coplotdom{g}]$ are two morphisms that satisfy the condition of (1), then $(dh_g)_{f(x)}=(dh_g')_{f(x)}$.
  \end{enumerate}
  Then there is an isomorphism of vector spaces:
  \[
    \righttangent{x}{X}\cong T_{f(x)} \coplotim{f}.
  \]
\end{cor}

\begin{proof}
This is straightforward from \thref{theorem right tangent space is right Kan extension}.
\end{proof}
\begin{cor}[\CW, Proposition 3.12] \label{corollary subspace external tangent}
  Let $X$ be a diffeological space, and let $x\in X$. Let $A$ be a $D$-open neighborhood of $x$ in $X$. Equip $A$ with a subdiffeology of $X$.
  Then the natural inclusion map induces an isomorphism $\external{x}{A}\cong \external{x}{X}$ and $\righttangent{x}{A}\cong \righttangent{x}{X}$.
\end{cor}

\begin{proof}
    $(A,x)$ and $(X,x)$ are isomorphic as objects of the category $\dflgloc$. 
\end{proof}

\begin{ex} \label{example cross right tangent}
  There is an isomorphism of vector spaces, where $\crossquot$ is defined in \exref{example definition cross_quot generating family} and 
  $\crosssub$ is defined in \exref{example crosssub internal tangent}:
  \[
    \righttangent{x}{\crossquot}\cong\righttangent{x}{\crosssub}\cong
    \begin{cases}
      \mathbb{R}^2 & (x=(0,0))\\
      \mathbb{R} & (x\neq (0,0)).
    \end{cases}
  \]
\end{ex}

\begin{proof}
The inclusion map $\crossquot\to\mathbb{R}^2;(a,b)\mapsto (a,b)$ and $\crosssub\to\mathbb{R}^2;(a,b)\mapsto (a,b)$ satisfy the conditions of \corref{corollary calculation of right tangent}.
\end{proof}

\subsection{Natural transformations from internal tangent spaces}\label{subsection natural transformation}
Thanks to the universality of the left Kan extension, there exists a natural transformation from the internal tangent functor to the right tangent functor or the external tangent functor. 
\begin{defi} \label{definition natural transformation}
  Let $X$ be a diffeological space, and let $x\in X$. We define a linear map $\nattrans_{(X,x)}\colon\internal{x}{X}\to\righttangent{x}{X}$ as follows:
  \[
    \nattrans_{(X,x)}([p,v])=\{T(g\circ p|_{p^{-1}(\coplotdom{g})})v\}_{[g,\coplotdom{g}]\in\obj(\coplotcat)}
  \]
  where $[p,v]\in\internal{x}{X}$, $p\colon \plotdom{p}\to X$ is a plot and $v\in T_0 (\plotdom{p})$. 
  This defines the natural transformation $\nattrans=\{\nattrans_{(X,x)}\}_{(X,x)\in\obj(\dflgloc)}\colon\internalfunctor\Rightarrow\rightfunctor$.
\end{defi}

Next, we prove that the image of $\nattrans_{(X,x)}$ is contained in $\external{x}{X}$.

\begin{lemma} \label{lemma differential operator is smooth}
  Let $W$ be an open set of $\mathbb{R}^n$. The map 
  \[
    \del{i}\colon C^{\infty}(W,\mathbb{R})\to \mathbb{R};f\mapsto \frac{\partial}{\partial x^i}f
  \]
  is smooth. In particular, for each differential operator $v$ of $W$, the map 
  \[
    v\colon C^{\infty}(W,\mathbb{R})\to \mathbb{R};f\mapsto v(f)
  \]
  is also smooth.
\end{lemma}

\begin{proof}
Let $p\colon \plotdom{p}\to C^{\infty}(W,\mathbb{R})$ be a plot, and let $m$ be the dimension of $\plotdom{p}$. 
Because of the definition of the functional diffeology on $C^{\infty}(W,\mathbb{R})$, the map $\ad(p)\colon\plotdom{p}\times W\to\mathbb{R}$ is smooth (\defref{definition functional diffeology}).
For each $u\in \plotdom{p}$, $\del{i}\circ p(u)$ is the $(m+i)$-th partial derivative of $\ad(p)$. 
Therefore, for any $u\in\plotdom{p}$ and $x\in W$, the following equality holds:
\[
  \ad\left(\del{i}\circ p\right)(u,x)=\del{m+i}(\ad(p))(u,x).
\]
Since $\ad(p)$ is smooth, the $(m+i)$-th partial derivative of $\ad(p)$ is also smooth. 
In other words, $\del{m+i}(\ad(p))$ is smooth, so $\ad(\del{i}\circ p)$ is also smooth because of the above equality. 
By the definition of the functional diffeology again, $\del{i}\circ p$ is smooth and this proves the first assertion of this lemma. 

It is straightforward to prove the second assertion by using the first assertion because every differential operator $v$ of $W$ is a linear combination of $\del{i}$'s.
\end{proof}

\begin{prop} \label{proposition image of internal is external}
  Let $X$ be a diffeological space, and let $x\in X$. Then $\nattrans_{(X,x)}(\internal{x}{X})$ is contained in $\external{x}{X}$.
  Therefore, we have the natural transformation $\nattransexternal\colon\internalfunctor\Rightarrow\externalfunctor$.
\end{prop}

\begin{proof}
Let $[p,v]\in\internal{x}{X}$, where $p\colon \plotdom{p}\to X$ is a plot and $v\in T_0 (\plotdom{p})$. 
We define $D=\nattrans_{(X,x)}([p,v])$. By regarding $D\in\righttangent{x}{X}$ as a linear map $\germ(X,x)\to\mathbb{R}$ which satisfies the Leibniz rule, 
the next equality holds:
\[
  D([g,\coplotdom{g}])=(T(g\circ p|_{p^{-1}(\coplotdom{g})})v)([{\id}_{\mathbb{R}},\mathbb{R}])=v([g\circ p|_{p^{-1}(\coplotdom{g})},p^{-1}(\coplotdom{g})]),
\]
where $v$ is regarded as a right tangent vector of $\righttangent{0}{\plotdom{p}}$. 
We should prove that $D\in\righttangent{x}{X}$ is an element of $\external{x}{X}$, in other words, $D$ is smooth as a map $\germ(X,x)\to\mathbb{R}$.
By the universality of the sum diffeology and quotient diffeology, 
it is enough to show that for any $D$-open neighborhood $B$ of $x$, the map $D_B\colon\smoothmap{B}{\mathbb{R}}\to \mathbb{R};g\mapsto D([g,B])$ is a smooth map. 
For $g\in C^{\infty}(B,\mathbb{R})$, we have
\[
  D_B(g)=D([g,B])=v([g\circ p|_{p^{-1}(B)},p^{-1}(B)]).
\]
Therefore, $D_B$ is a composition of two maps $C^{\infty}(B,\mathbb{R})\to C^{\infty}(p^{-1}(B),\mathbb{R});g\mapsto g\circ p|_{p^{-1}(B)}$ 
and $C^{\infty}(p^{-1}(B),\mathbb{R})\to\mathbb{R};h\mapsto v([h,\coplotdom{h}])$. 
Since the second map is smooth because of \lemref{lemma differential operator is smooth}, $D_B$ is smooth.
\end{proof}

\begin{remark}{}{}
  An element of the image of $\nattrans_{(X,x)}$ is contained in the image of $\rightfunctor(f)\colon\righttangent{0}{\mathbb{R}}\to\righttangent{x}{X}$, 
  where $f\colon \mathbb{R}\to X$ is some smooth map (because $\righttangent{0}{\mathbb{R}}\cong \internal{0}{\mathbb{R}}$).
  Therefore, this proposition is immediately proved by \proref{proposition external tangent forms functor}.

  We can also prove this proposition by the universality of the left Kan extension.
  First, the universality of the left Kan extension, it follows that there exists the unique natural transformation $\internalfunctor\Rightarrow\externalfunctor$. 
  If we composite this natural transformation with the inclusion natural transformation $\externalfunctor\Rightarrow\rightfunctor$, we have a natural transformation $\internalfunctor\Rightarrow\rightfunctor$. 
  However, by the universality of the left Kan extension again, this natural transformation coincides with the natural transformation $\nattrans\colon\internalfunctor\Rightarrow\rightfunctor$.
  This means that the image of the natural transformation $\nattrans$ is contained in $\externalfunctor$. 
\end{remark}

\begin{cor} \label{corollary right equal to external}
  Let $X$ be a diffeological space, and let $x\in X$. If the natural transformation map $\nattrans_{(X,x)}$ is surjective, then $\righttangent{x}{X}=\external{x}{X}$.
\end{cor}

\begin{proof}
Since all elements in the image of $\nattrans_{(X,x)}$ are external tangent vectors (\proref{proposition image of internal is external}), we obtain this corollary.
\end{proof}

\begin{remark}{(Vincent's tangent space)}\label{remark Vincent's tangent space}
  We also consider the image of this natural transformation as one of the types of tangent spaces of diffeological spaces.
  We denote this tangent space by $\vincenttangent{x}{X}$. 
  This tangent space is also a generalization of a tangent space of manifolds. 
  Very close definition to this is considered in \vincent. However, the definition of Vincent is considering differential operators only on global smooth functions rather than germs of them.
  We discuss the difference between them in the next subsection. 
\end{remark}

\subsection{Variants of external tangent spaces} \label{subsection variants of external}

\begin{defi} \label{definition other external tangent space}
  Let $X$ be a diffeological space, and let $x\in X$. We equip $C^{\infty}(X,\mathbb{R})$ with the functional diffeology and the standard $\mathbb{R}$-algebra structure. 
  A map $D\colon C^{\infty}(X,\mathbb{R})\to\mathbb{R}$ is called a \italic{global right tangent vector} if the following two conditions hold:
  \begin{description}
    \item[Linearity] $D$ is a linear map.
    \item[Leibniz rule] For all $g$, $h\in \smoothmap{X}{\mathbb{R}}$, $D(gh)=g(x)D(h)+h(x)D(g)$.
  \end{description}
  We denote the set of all global right tangent vectors by $\globalrighttangent{x}{X}$ and naturally we give the natural vector space structure to $\globalrighttangent{x}{X}$.
  This vector space is called the \italic{global right tangent space of $X$ at $x$}.
  Moreover, we call a global right tangent vector $D\in\globalrighttangent{x}{X}$ a \italic{global external tangent vector} if the following condition also holds:
  \begin{description}
    \item[Smoothness] $D$ is a smooth map.
  \end{description}
  The set $\globalexternal{x}{X}$ of all global external tangent vectors define the subspace of $\globalrighttangent{x}{X}$, and it is called the \italic{global external tangent space of $X$ at $x$}. 
\end{defi}

\begin{prop} \label{proposition global right tangent is right Kan extension}
  Let $X$ be a diffeological space, and let $x\in X$. Then there is an isomorphism of vector spaces:
  \[
    \Ran_{\inclbased} T^{\enclbased} (X,x)\cong \globalrighttangent{x}{X}.
  \]
\end{prop}

\begin{proof}
This is the same as the proof of \thref{theorem right tangent space is right Kan extension}. 
\end{proof}

Because of this proposition, we sometimes regard $\globalrighttangent{x}{X}$ as $\Ran_{\inclbased} T^{\enclbased} (X,x)$ and $\globalexternal{x}{X}$ as a subspace of $\Ran_{\inclbased} T^{\enclbased} (X,x)$.

\begin{prop} \label{proposition global right is differential invariant}
  Let $X$ and $Y$ be two diffeological spaces, and let $f\colon X\to Y$ be a smooth map. If $f\colon \Phi X\to\Phi Y$ is a diffeomorphism as a map between differential spaces (appendix \ref{section differential and frolicher}), 
  then there is an isomorphism of vector spaces: 
  \[
    \globalrighttangent{x}{X}\cong\globalrighttangent{f(x)}{Y}.
  \]
\end{prop}

\begin{proof}
For any diffeological space $(X,\mathscr{D}_X)$, there is a bijection between $C^{\infty}(X,\mathbb{R})$ and $\Phi\mathscr{D}_X$. 
Since this bijection is a linear isomorphism, 
we are done.
\end{proof}
  This proposition means that global right tangent space is a differential invariant. 
  However, we do not know whether the global external tangent space is a differential invariant.

We also consider the natural transformation from internal tangent spaces to global right tangent spaces (or global external tangent spaces). 

\begin{defi} \label{definition natural transformation to global tangent space}
  Let $X$ be a diffeological space, and let $x\in X$. We define a linear map $\nattransglobal_{(X,x)}\colon\internal{x}{X}\to\globalrighttangent{x}{X}$ as follows:
  \[
    \nattransglobal_{(X,x)}([p,v])=\{T(g\circ p)v\}_{g\in \smoothmap{X}{\mathbb{R}}}
  \]
  where $[p,v]\in\internal{x}{X}$, $p\colon \plotdom{p}\to X$ is a plot and $v\in T_0 (\plotdom{p})$. 
  This defines the natural transformation $\nattransglobal=\{\nattransglobal_{(X,x)}\}_{(X,x)\in\obj(\dflgbased)}\colon\internalfunctor\Rightarrow\globalrightfunctor$,
  where we regard $\internalfunctor$ as a functor from $\dflgbased$.
  We also have the natural transformation $\nattransglobalexternal\colon\internalfunctor\Rightarrow\globalexternalfunctor$ in the same manner as \proref{proposition image of internal is external}.   
\end{defi}

We want to know when the right tangent space and the global right tangent space coincide. Here is one of sufficient conditions.
In what follows, for a diffeological space $X$, we call the smallest topology which makes all elements of $\smoothmap{X}{\mathbb{R}}$ continuous the \italic{$I$-topology} of $X$. 
This coincides with the $I$-topology of the differential space $\Phi X$, which is introduced in the appendix \ref{section differential and frolicher}. 
Note that the $I$-topology is contained in the $D$-topology of the pullback diffeology of $X$ by the family of smooth functions $C^\infty(X, \mathbb{R})$. 
The inclusion need not be an equality. For example, if $X=\mathbb{R}^{J}$ is equipped with the product diffeology and $|J|=2^{\aleph_0}$, then its $I$-topology is the product topology, whereas its $D$-topology is the strictly finer $c^\infty$-topology; see \cite[Theorem~3.7]{MR3270173} and \cite[Example~4.8 and the proof of Example~4.27]{MR1471480}.

\begin{defi} \label{definition smoothly regular}
  Let $X$ be a diffeological space, let $x\in X$, and let $\mathscr{O}$ be a topology on the set $X$. $X$ is said to be \italic{smoothly regular at $x$ for $\mathscr{O}$} 
  if for any open neighborhood $U\in\mathscr{O}$ of $x$, there exists some smooth map $\lambda\colon X\to \mathbb{R}$ such that 
  $\lambda(x)=1$ and $\lambda(y)=0$ for any $y\notin U$. 
  Also, $X$ is said to be \italic{smoothly regular for $\mathscr{O}$} if for any $x\in X$, $X$ is smoothly regular at $x$ for $\mathscr{O}$.
\end{defi}

If the topology $\mathscr{O}$ in this definition is the $I$-topology of $X$, 
the diffeological space $X$ automatically becomes smoothly regular. 
More strongly, the following proposition holds. 

\begin{prop} \label{proposition diffeological space is strong smoothly regular w.r.t. initial topology}
  Let $X$ be a diffeological space, let $x\in X$, and let $\mathscr{O}$ be the $I$-topology of $X$.
  Then for any open neighborhood $U\in\mathscr{O}$ of $x$, 
  there exists some open neighborhood $V\in\mathscr{O}$ of $x$ and some smooth map $\lambda\colon X\to \mathbb{R}$ such that 
  $\lambda(y)=1$ for any $y\in V$ and $\lambda(y)=0$ for any $y\notin U$. 
  In particular, $X$ is smoothly regular at $x$ for its $I$-topology.
\end{prop}

\begin{proof}
  This proof is almost the same as \BKW, Proposition 6.3, but we give the proof. 
  Let $x\in X$, and let $\mathscr{D}_X$ be the diffeology of $X$. 
  Then by the definition of $I$-topology, there exist functions $h_1,\dots,h_k\in \Phi(\mathscr{D}_X)$ and open intervals $I_1,\dots, I_k$ such that 
  $x\in \bigcap_{i=1}^k h_i^{-1}(I_i) \subset U$. Take $\lambda=b\circ (h_1,\dots,h_k)$ where $b\colon\mathbb{R}^k\to[0,1]$ is a smooth bump function 
  whose support is contained in $I_1\times\dots\times I_k$ and such that $b(y)=1$ for $y\in W$, where $W$ is some open neighborhood of $(h_1(x),\dots,h_k(x))$ contained in $I_1\times,\dots,\times I_k$. 
  Take $V=(h_1,\dots,h_k)^{-1}(W)$. This is an $I$-open set and with these $\lambda$ and $V$, the condition of this proposition is satisfied. 
\end{proof}

\begin{lemma}[\BKW, Lemma 6.1] \label{lemma I-topology contained in D-topology}
  Let $X$ be a diffeological space. If $U$ is an $I$-open set of $X$, then $U$ is a $D$-open set of $X$. 
\end{lemma}

\begin{prop} \label{proposition equivalent condition smoothly regular at point}
  Let $X$ be a diffeological space, and let $x\in X$. Then the following conditions are equivalent:
  \begin{description}
    \item[1] $X$ is smoothly regular at $x$ for its $D$-topology.
    \item[2] For any $D$-open neighborhood $U$ of $x$, 
    there exists some $D$-open neighborhood $V$ of $x$ and some smooth map $\lambda\colon X\to \mathbb{R}$ such that 
    $\lambda(y)=1$ for $y\in V$ and $\lambda(y)=0$ for $y\notin U$. 
    \item[3] For any $D$-open neighborhood $U$ of $x$, there exists some $I$-open neighborhood $V$ of $x$ such that $V\subset U$ holds. 
  \end{description}
\end{prop}

\begin{proof}
$2\Rightarrow 1$ is trivial.

We prove $3\Rightarrow 2$. Let $x\in X$, and let $U$ be a $D$-open neighborhood of $x$.
Then by the condition 3, take an $I$-open neighborhood $W$ of $x$ such that $W\subset U$. 
By \proref{proposition diffeological space is strong smoothly regular w.r.t. initial topology},
there exists some $I$-open neighborhood $V$ of $x$ for and some smooth map $\lambda\colon X\to \mathbb{R}$ such that 
$\lambda(y)=1$ for $y\in V$ and $\lambda(y)=0$ for $y\notin W$. Because of \lemref{lemma I-topology contained in D-topology}, $V$ is a $D$-open neighborhood of $x$.
Therefore, the condition 2 holds. 

Finally, we prove $1\Rightarrow 3$. Let $U$ be a $D$-open neighborhood of $x$. 
Since $X$ is smoothly regular at $x$ for $D$-topology, there exists some smooth map $\lambda\colon X\to \mathbb{R}$ such that $\lambda(x)=1$ and $\lambda(y)=0$ for $y\notin U$. 
Set $V=\lambda^{-1}((1/2,\infty))$. This is an $I$-open neighborhood of $x$ which is contained in $U$. 
\end{proof}

\begin{cor} \label{corollary equivalent condition global smoothly regular}
  Let $X$ be a diffeological space. Then the following conditions are equivalent:
  \begin{description}
    \item[1] $X$ is smoothly regular for its $D$-topology.
    \item[2] For any $x\in X$ and any $D$-open neighborhood $U$ of $x$, 
    there exists some $D$-open neighborhood $V$ of $x$ and some smooth map $\lambda\colon X\to \mathbb{R}$ such that $V\subset U$,
    $\lambda(y)=1$ for $y\in V$ and $\lambda(y)=0$ for $y\notin U$. 
    \item[3] The $D$-topology of $X$ and the $I$-topology of $X$ coincide. 
  \end{description}
\end{cor}

\begin{proof}
This is straightforward from \proref{proposition equivalent condition smoothly regular at point}. 
\end{proof}

In what follows, we simply say that a diffeological space $X$ is smoothly regular at the point $x\in X$ if it is smoothly regular at $x$ for its $D$-topology.
Also, a diffeological space which satisfies the second condition of this corollary is called regular diffeological space in \vincent. 
Therefore, we sometimes call a smoothly regular space a regular diffeological space in Vincent's meaning, 
especially when we want to emphasize that the diffeological space satisfies the second condition of this corollary or the previous proposition.

We prove that the global right tangent space of a smoothly regular diffeological space is isomorphic to the standard right tangent space.

\begin{lemma}[\vincent, Lemma 3.5.2] \label{lemma 0 extension is smooth}
  Let $U\subset \mathbb{R}^n$ be an open set, and let $b\colon\mathbb{R}^n\to \mathbb{R}$ be a smooth function being zero outside $U$.
  And let $f\colon U\to\mathbb{R}$ be any function. If $X$ is smoothly regular at $x$, then the function
  \[
    \mathbb{R}^n\to\mathbb{R};\xi\mapsto
    \begin{cases}
      b(\xi)f(\xi) & (\xi\in U)\\
      0 & (\xi\notin U).
    \end{cases}
  \]
  is smooth.
\end{lemma}

\begin{lemma} \label{lemma existence of 0 extension dflg space}
  Let $X$ be a diffeological space, let $x\in X$, and let $g\colon \coplotdom{g}\to \mathbb{R}$ be a smooth map where $\coplotdom{g}$ is a $D$-open set of $X$.
  Then there exists some $D$-open set $V$, some smooth map $\tilde{g}\colon X\to \mathbb{R}$ such that $V\subset U$ and 
  $\tilde{g}(y)=g(y)$ for any $y\in V$ and $\tilde{g}(y)=0$ for any $y\notin \coplotdom{g}$. 
\end{lemma}

\begin{proof}
Let $g\colon \coplotdom{g}\to \mathbb{R}$ be a smooth function where $\coplotdom{g}$ is a $D$-open set of $X$.
Since $X$ is regular at $x$ in Vincent's meaning, there exist some $D$-open neighborhood $V$ of $x$ and some smooth map $\lambda\colon X\to \mathbb{R}$ such that $V\subset \coplotdom{g}$,
$\lambda(y)=1$ for $y\in V$ and $\lambda(y)=0$ for $y\notin \coplotdom{g}$. We define a smooth map $\tilde{g}\colon X\to\mathbb{R}$ as follows:
\[
  \tilde{g}(y)=
  \begin{cases}
    \lambda(y)g(y) & (y\in \coplotdom{g})\\
    0 & (y\notin \coplotdom{g}).
  \end{cases}
\]
Indeed, this function $\tilde{g}$ is smooth because for any plot $p\colon \plotdom{p}\to X$ of X and any $u\in \plotdom{p}$, it holds that
\[
  \tilde{g}\circ p(u)=
  \begin{cases}
    \lambda(p(u))g(p(u)) & (u\in p^{-1}(\coplotdom{g}))\\
    0 & (u\notin p^{-1}(\coplotdom{g})).
  \end{cases}
\] 
Since $\coplotdom{g}$ is a $D$-open set of $X$, $p^{-1}(\coplotdom{g})$ is an open set of $\plotdom{p}$. 
Therefore, by \lemref{lemma 0 extension is smooth}, it follows that $\tilde{g}\circ p$ is smooth and this proves that $\tilde{g}$ is smooth.
\end{proof}
\begin{thm} \label{theorem global right tangent space of regular space}
  Let $X$ be a diffeological space, and let $x\in X$. If $X$ is smoothly regular at $x$, then there is an isomorphism of vector spaces:
  \[
    \globalrighttangent{x}{X}\cong\righttangent{x}{X}.
  \]
\end{thm}

\begin{proof}
  First, we define a linear map $\alpha\colon\righttangent{x}{X}\to\globalrighttangent{x}{X}$. Let $D\in \righttangent{x}{X}$, and let $g\in \smoothmap{X}{\mathbb{R}}$. 
  We define $\alpha(D)$ as follows:
  \[
    \alpha(D)(g)=D([g,X]).
  \]
  Second, we define a linear map $\beta\colon\globalrighttangent{x}{X}\to\righttangent{x}{X}$. Let $D\in \globalrighttangent{x}{X}$, and let $G=[g,\coplotdom{g}]\in\germ(X,x)$. 
  Since $X$ is regular at $x$ in Vincent's meaning, there exist some $D$-open neighborhood $V$ of $x$ and some smooth map $\tilde{g}\colon X\to \mathbb{R}$ 
  such that $V\subset U$, $\tilde{g}(y)=g(y)$ for $y\in V$, and $\tilde{g}(y)=0$ for $y\notin \coplotdom{g}$ (\lemref{lemma existence of 0 extension dflg space}).
  We want to define $\beta(D)$ as follows:
  \[
    \beta(D)([g,\coplotdom{g}])=D(\tilde{g}).
  \]
  We need to show that this definition is independent to the choice of extended function $\tilde{g}$. To show this, take two functions $\tilde{g}_1$, $\tilde{g}_2$ 
  and two $D$-open sets $V_1$, $V_2$ which satisfy the condition of \lemref{lemma existence of 0 extension dflg space}. 
  We set $h=\tilde{g}_1-\tilde{g}_2$ and $W=V_1\cap V_2$. Since $X$ is regular at $x$ in Vincent's meaning, there exist some $D$-open neighborhood $O$ of $x$ and some smooth map $\lambda\colon X\to \mathbb{R}$ 
  such that $O\subset W$, $\lambda(y)=1$ for $y\in O$, and $\lambda(y)=0$ for $y\notin W$ (\proref{proposition equivalent condition smoothly regular at point}).
  Since the product function $\lambda h$ is equal to the zero function at every point of $X$, the following equality holds:
  \[
    0=D(\lambda h)=\lambda(x)D(h)+h(x)D(\lambda)=1\cdot D(h)+0\cdot D(\lambda)=D(h).
  \]
  The second equality holds because of the Leibniz rule. Since $D$ is a linear map, it follows that $D(\tilde{g}_1)=D(\tilde{g}_2)$ and this is what we wanted to prove.

  We prove that $\beta\circ\alpha={\id}_{\righttangent{x}{X}}$ and $\alpha\circ\beta={\id}_{\globalrighttangent{x}{X}}$.
  First, let $D\in \righttangent{x}{X}$, and let $[g,\coplotdom{g}]\in\germ(X,x)$. 
  We denote one of extensions of $g$ in \lemref{lemma existence of 0 extension dflg space} by $\tilde{g}$ and take $V$ a $D$-open neighborhood of $x$ such that ${\tilde{g}|}_V={g|}_V$. 
  The following equality holds:
  \[
    \beta\circ\alpha(D)([g,\coplotdom{g}])=\alpha(D)(\tilde{g})=D([\tilde{g},X])=D([g,\coplotdom{g}]).
  \]
  The third equality holds because ${\tilde{g}|}_V={g|}_V$ holds. Therefore, $\beta\circ\alpha={\id}_{\righttangent{x}{X}}$ holds.
  Second, let $D\in \globalrighttangent{x}{X}$, and let $g\in\smoothmap{X}{\mathbb{R}}$. The following equality holds:
  \[
    \alpha\circ \beta(D)(g)=\beta(D)([g,X])=D(g).
  \]
  The second equality holds because $g$ itself is one of extensions in \lemref{lemma existence of 0 extension dflg space}.
  Therefore, $\alpha\circ\beta={\id}_{\globalrighttangent{x}{X}}$ holds and this completes the proof.
\end{proof}

\begin{thm} \label{theorem global external tangent space of regular space}
  Let $X$ be a diffeological space, and let $x\in X$. If $X$ is smoothly regular at $x$, then there is an isomorphism of vector spaces:
  \[
    \globalexternal{x}{X}\cong\external{x}{X}.
  \]
  This isomorphism is induced from the isomorphism of \thref{theorem global right tangent space of regular space}.
\end{thm}

\begin{proof}
We use maps $\alpha$, $\beta$ in the proof of \thref{theorem global right tangent space of regular space}.
Let $D\in \globalexternal{x}{X}$. We first prove that $\beta(D)\in \external{x}{X}$. 
Let $B$ be a $D$-open neighborhood of $x$. Take a bump function $\lambda\colon X\to \mathbb{R}$ and a $D$-open set $V$ in the second condition of \proref{proposition equivalent condition smoothly regular at point}. 
For any smooth function $g\colon B\to \mathbb{R}$, we define the function $\tilde{g}$ as follows:
\[
  \tilde{g}(y)=
  \begin{cases}
    \lambda(y)g(y) & (y\in B)\\
    0 & (y\notin B).
  \end{cases}
\]
This function is smooth as in the proof of \lemref{lemma existence of 0 extension dflg space}.
Using this notation, it follows that $\beta(D)([g,B])=D(\tilde{g})$ for any function $g\colon B\to \mathbb{R}$.
Therefore, we need to show that the map $D_B\colon\smoothmap{B}{\mathbb{R}}\to \mathbb{R};g\mapsto D(\tilde{g})$ is smooth 
(by the universality of the sum diffeology and the quotient diffeology).
This map is the composition of $\smoothmap{B}{\mathbb{R}}\to\smoothmap{X}{\mathbb{R}};g\mapsto\tilde{g}$ and $D$. 
$D$ is a smooth map because $D$ is in $\globalexternal{x}{X}$. The first map is also smooth by the next \lemref{lemma extension map is smooth}. 
Therefore, it follows that $\beta(D)$ is in $\external{x}{X}$.

Second, let $D\in \external{x}{X}$. We prove that $\alpha(D)\in \globalexternal{x}{X}$. 
For any $g\in\smoothmap{X}{\mathbb{R}}$, we have $\alpha(D)(g)=D([g,X])$. So we need to show that $\smoothmap{X}{\mathbb{R}}\to\mathbb{R};g\mapsto D([g,X])$ is smooth, 
but this is smooth because of the fact that $D\in \external{x}{X}$ and by the universality of the sum diffeology and the quotient diffeology. 
Therefore, it follows that $\alpha(D)$ is in $\globalexternal{x}{X}$. 
With those facts, we have the isomorphism $\globalexternal{x}{X}\cong\external{x}{X}$ by restricting $\alpha$ and $\beta$.
\end{proof}

\begin{lemma} \label{lemma extension map is smooth}
  Let $X$ be a diffeological space, and let $x\in X$. If $X$ is smoothly regular at $x$, 
  then the map $\smoothmap{B}{\mathbb{R}}\to\smoothmap{X}{\mathbb{R}};g\mapsto\tilde{g}$ defined in the proof of \thref{theorem global external tangent space of regular space} is smooth.
\end{lemma}

\begin{proof}
We denote the map $\smoothmap{B}{\mathbb{R}}\to\smoothmap{X}{\mathbb{R}};g\mapsto\tilde{g}$ by $E$. 
Let $p\colon \plotdom{p}\to \smoothmap{B}{\mathbb{R}}$ be a plot of $\smoothmap{B}{\mathbb{R}}$. We need to prove that $E\circ p\colon \plotdom{p}\to \smoothmap{X}{\mathbb{R}}$ is smooth.
Namely, we need to prove that $\ad(E\circ p)\colon\plotdom{p}\times X\to\mathbb{R}$ is smooth (\defref{definition functional diffeology}). 
To prove this, let $q=(q_1,q_2)\colon\plotdom{q}\to \plotdom{p}\times X$ be a plot. It is enough to show that $\ad(E\circ p)\circ q$ is smooth. 
For any $u\in \plotdom{q}$, the following equality holds:
\begin{equation*}
  \begin{split}
    (\ad(E\circ p)\circ q)(u)&=\widetilde{p(q_1(u))}(q_2(u))=
    \begin{cases}
      \lambda(q_2(u))\cdot p(q_1(u))(q_2(u)) & (u\in q_2^{-1}(\coplotdom{g}))\\
      0 & (u\notin q_2^{-1}(\coplotdom{g}))
    \end{cases}\\
    &=\begin{cases}
      \lambda(q_2(u))\cdot \ad(p)\circ q(u) & (u\in q_2^{-1}(\coplotdom{g}))\\
      0 & (u\notin q_2^{-1}(\coplotdom{g})).
    \end{cases}
  \end{split}
\end{equation*}
By \lemref{lemma 0 extension is smooth} and the fact that $\ad(p)$ is smooth (\defref{definition functional diffeology}), $\ad(E\circ p)\circ q$ is also smooth. Therefore, it follows that $E\circ p$ is smooth and this means that $E$ is smooth.
\end{proof}

There is some diffeological space that is not smoothly regular. Here is one of the examples.

\begin{ex} \label{example definition R/open interval}
  Let $\equivrel$ be the equivalence relation on $\mathbb{R}$ such that $x\equivrel y\Longleftrightarrow x=y$ or $x$ and $y$ are both elements of $(0,1)$. 
  We denote the quotient diffeological space $\mathbb{R}/{\equivrel}$ by $\mathbb{R}/(0,1)$ and we denote the point of $\mathbb{R}/(0,1)$ which is represented by $t\in \mathbb{R}$ by $[t]$. 
  $\mathbb{R}/(0,1)$ is a diffeological space which is not \regular at $[1/2]$. 
\end{ex}

\begin{proof}
Since the $D$-topology of $\mathbb{R}/(0,1)$ is the quotient topology from $\mathbb{R}$ (\IZ, 2.12), 
the set $\{[1/2]\}$ is a $D$-open set of $\mathbb{R}/(0,1)$.
Suppose that there exist some $D$-open neighborhood $V$ of $x$ and some smooth map $\lambda\colon X\to\mathbb{R}$ such that $V\subset \{[1/2]\}$, $\lambda(x)=1$ for $x\in V$, and $\lambda(x)=0$ for $x\notin \{[1/2]\}$. 
Since $V=\{[1/2]\}$ is the unique $D$-open set which satisfies $x\in V\subset \{[1/2]\}$, 
we have to take $V=\{[1/2]\}$ and $\lambda\colon X\to\mathbb{R}$ is forced to be the function 
such that $\lambda([1/2])=1$ and $\lambda(x)=0$ if $x\neq [1/2]$. 
However, $\lambda$ is not smooth because $\lambda\circ \pi$ is not smooth, 
where $\pi\colon \mathbb{R}\to\mathbb{R}/(0,1)$ is the natural projection.
\end{proof}

\begin{ex} \label{example right tangent space R/open interval}
  There is an isomorphism of vector spaces:
  \[
    \righttangent{[1/2]}{\mathbb{R}/(0,1)}=\external{[1/2]}{\mathbb{R}/(0,1)}\cong 0.
  \]
\end{ex}

\begin{proof}
Since $\{[1/2]\}$ is a $D$-open set of $\mathbb{R}/(0,1)$, we have $\righttangent{[1/2]}{\mathbb{R}/(0,1)}\cong \righttangent{[1/2]}{\{[0,1]\}}\cong 0$ 
(the subdiffeology on $\{[0,1]\}\subset \mathbb{R}/(0,1)$ is diffeomorphic to $\mathbb{R}^0$, so its tangent space is trivial).
\end{proof}

We do not know whether the global right tangent space of $\mathbb{R}/(0,1)$ at $[1/2]$ is also isomorphic to zero. 
However, I suspect that $\globalrighttangent{[1/2]}{\mathbb{R}/(0,1)}\cong 0$ holds. 
We will explain why we believe this.
We have the next lemma:

\begin{lemma} \label{lemma splitting square root global right tangent space}
  Let $X$ be a diffeological space, let $x\in X$, and let $D\in\globalrighttangent{x}{X}$. 
  Given any $f\in\smoothmap{X}{\mathbb{R}}$ suppose that there exist two smooth maps $g$, $h\in\smoothmap{X}{\mathbb{R}}$ which satisfy $g(x)=h(x)=0$ and $f=f(x)+gh$, where $f(x)\colon X\to \mathbb{R};y\mapsto f(x)$ is a constant map. 
  Then $D(f)=0$. 
\end{lemma}

\begin{proof}
By the linearity and Leibniz rule, we have 
\[
  D(f)=D(f(x))+D(gh)=g(x)D(h)+h(x)D(g)=0.
\]
Indeed, for the constant map $1\colon X\to \mathbb{R};y\mapsto 1$, we have $D(1)=1(x)D(1)+1(x)D(1)=2D(1)$. Therefore, it follows that $D(f(x))=f(x)D(1)=0$ by the linearity of $D$.
\end{proof}

From here, we denote the natural projection $\mathbb{R}\to\mathbb{R}/(0,1)$ by $\pi$. 
Also, for a smooth map $f\colon \mathbb{R}/(0,1)\to \mathbb{R}$, we denote the smooth map $f\circ \pi$ by $\hat{f}$. 

\begin{ex} \label{example global right tangent space R/open interval}
  Take an element $D\in \globalrighttangent{[1/2]}{\mathbb{R}/(0,1)}$ and a smooth map $f\in\smoothmap{\mathbb{R}/(0,1)}{\mathbb{R}}$. 
  If $\hat{f}\in\smoothmap{\mathbb{R}}{\mathbb{R}}$ is in the linear span (we denote this linear span by $S$) of the set 
  \[
    \{gh+a1\in\smoothmap{\mathbb{R}}{\mathbb{R}}\mid \text{$g$,$h\in \smoothmap{\mathbb{R}}{\mathbb{R}}$, $g((0,1))=h((0,1))=\{0\}$, and $a\in \mathbb{R}$.}\}\subset \smoothmap{\mathbb{R}}{\mathbb{R}},
  \]
  then $D(f)=0$ holds by the universality of the quotient diffeology and \lemref{lemma splitting square root global right tangent space}. 
  In particular, if $\hat{f}\in\smoothmap{\mathbb{R}}{\mathbb{R}}$ is a non-negative function and its square root is also smooth, then $D(f)=0$ holds. 
  However, even if in this case, it is known that there is a smooth function $\hat{f}$ whose square root is not smooth (Bony-Broglia-Colombini-Pernazza \nonnegativeroot).
  This is such a function:
  \[
    \hat{f}(t+1)=\begin{cases}
      e^{-1/t}\left(\sin^2(\pi/t)+e^{-1/t^2}\right)& (t\geq 0)\\
      0 & (t\leq 0)
    \end{cases}.
  \]
  Here, we horizontally translated the function in \nonnegativeroot, Section 2 by 1 to adjust it to our settings. 
  Therefore, it is difficult to prove that any $\hat{f}\in\smoothmap{\mathbb{R}}{\mathbb{R}}$ is in $S$. 
  On the other hand, for any $\varepsilon>0$, there is a smooth bump function $\lambda_{\varepsilon}:\mathbb{R}\to\mathbb{R}$ such that $\lambda_{\varepsilon}((0,1))=\{1\}$ and 
  $\lambda_{\varepsilon}(t)=0$ for any $t\leq -\varepsilon$ and $t\geq 1+\varepsilon$. 
  This indicates that for any $f\in\smoothmap{\mathbb{R}/(0,1)}{\mathbb{R}}$, if $\hat{f}$ is equal to some element of $S$ 
  only within some interval $(-\varepsilon,1+\varepsilon)$, then $D(f)=0$ holds. 
  I expect that almost all elements $f\in\smoothmap{\mathbb{R}/(0,1)}{\mathbb{R}}$ satisfy this condition. This is why I suspect that $\globalrighttangent{[1/2]}{\mathbb{R}/(0,1)}\cong 0$ holds.
\end{ex}

\begin{ex} \label{example manifold external tangent}
  Let $M$ be a manifold, and let $x\in M$. Then the right tangent space is isomorphic to the standard tangent space 
  because $(M,x)$ is isomorphic to $(U,0)$ as an object of the category $\dflgloc$, where $U$ is a local chart of $M$ at $x$, and $\rightfunctor$ is a strict extension of the tangent functor from $\enclzeroloc$. 
  Moreover, since all elements of $\righttangent{x}{M}$ are in the image of the natural transformation $\nattrans_{(M,x)}$, the external tangent space is also isomorphic to the standard tangent space (\proref{proposition image of internal is external}). 
  Since $M$ is a smoothly regular diffeological space, the global right or external tangent spaces are also the same:  
  \[
    \righttangent{x}{M}=\external{x}{M}\cong\globalrighttangent{x}{M}=\globalexternal{x}{M}\cong T_x(M).
  \]
\end{ex}

\begin{lemma} \label{lemma locally smooth function globalexternal}
  Let $X$ be a diffeological space, and let $x\in X$. If any smooth function $X\to \mathbb{R}$ is locally constant for the $D$-topology of $X$, then there is an isomorphism of vector spaces: 
  \[
    \globalrighttangent{x}{X}=\globalexternal{x}{X}\cong 0.
  \]
\end{lemma}

\begin{proof}
  Let $D\in\globalrighttangent{x}{X}$, and let $\chi_U\colon X\to \mathbb{R}$ be a characteristic function where $U\subset X$ is a $D$-connected component of $x$ (since $\chi_U$ is locally constant, $\chi_U$ is smooth).
  Because of the Leibniz rule, we have
  \[
    D(\chi_U)=\chi_U(x)D(\chi_U)+\chi_U(x)D(\chi_U).
  \] 
  Since $x\in U$, the right side of this equality is equal to $2D(\chi_U)$ and this indicates that $D(\chi_U)=0$. 
  Next, let $f\colon X\to \mathbb{R}$ be a smooth function. By the Leibniz rule again, we have
  \[
    D(f\chi_U)=f(x)D(\chi_U)+\chi_U(x)D(f)=f(x)\cdot 0+1\cdot D(f)=D(f).
  \]
  However, since $f\chi_U=f(x)\chi_U$ holds, 
  we have $D(f\chi_U)=0$. 
  Therefore, $D(f)=0$, and this means that $D=0$. 
\end{proof}

A similar lemma holds for the standard right tangent space and the external tangent space. The next lemma is used in \CW, Example 3.15.

\begin{lemma}[\CW, Example 3.15] \label{lemma locally smooth function external}
  Let $X$ be a diffeological space, and let $x\in X$. 
  If any smooth function $A\to \mathbb{R}$ is locally constant for the $D$-topology of $X$ for any $D$-open set $A$ of $X$, then there is an isomorphism of vector spaces: 
  \[
    \righttangent{x}{X}=\external{x}{X}\cong 0.
  \]
\end{lemma}

\begin{proof}
  The proof of this lemma is the same as the previous \lemref{lemma locally smooth function globalexternal}. 
\end{proof}

\begin{ex}[\CW, Example 3.14] \label{example max/min external tangent}
  Let $X$ be a set, and let $x\in X$. We denote the diffeological space $(X,\mathscr{D}_X^{max})$ by $X_{max}$ and $(X,\mathscr{D}_X^{min})$ by $X_{min}$
  ($\mathscr{D}_X^{max}$ and $\mathscr{D}_X^{min}$ were defined in \exref{example definition max/min diffeology}). 
  Since the $D$-topology and the $I$-topology of $X_{max}$ are both indiscrete, $X_{max}$ is smoothly regular. Also, since the $D$-topology and the $I$-topology of $X_{min}$ are both discrete, $X_{min}$ is smoothly regular. 
  Therefore, we have $\globalrighttangent{x}{X_{max}}\cong \righttangent{x}{X_{max}}$ and $\globalrighttangent{x}{X_{min}}\cong \righttangent{x}{X_{min}}$. 
  Moreover, from the previous \lemref{lemma locally smooth function globalexternal}, there is an isomorphism of vector spaces:
  \[
    \righttangent{x}{X_{max}}=\external{x}{X_{max}}\cong 0,\;\righttangent{x}{X_{min}}=\external{x}{X_{min}}\cong 0. 
  \]
  Therefore, all variants of tangent spaces of $X_{max}$ and $X_{min}$ are isomorphic to zero.
\end{ex}

\begin{ex} \label{example topological space external tangent}
  Let $T$ be a topological space, and let $x\in T$. We think of the diffeological space $\conti(T)$ defined in \exref{example definition topological space diffeology}. 
  As proved in \CW, Example 3.15, it is known that for any $D$-open set $A\subset X$, all smooth functions $A\to \mathbb{R}$ are locally constant. 
  Therefore, there is an isomorphism of vector spaces (\lemref{lemma locally smooth function globalexternal}, \lemref{lemma locally smooth function external}):
  \[
    \righttangent{x}{\conti(T)}=\external{x}{\conti(T)}\cong\globalrighttangent{x}{\conti(T)}=\globalexternal{x}{\conti(T)}\cong 0.
  \]
\end{ex}

\begin{remark}{}{}
  For almost all topological spaces $(T,\mathscr{O}_T)$, $\conti(T)$ is not smoothly regular. We see this here.
  First, $\mathscr{O}_T\subset \mathscr{O}_{\conti(T)}$ where we denote the $D$-topology of $\conti(T)$ by $\mathscr{O}_{\conti(T)}$. 
  Indeed, for any $U\in\mathscr{O}_T$ and any plot $p\colon \plotdom{p}\to \conti(T)$, $p^{-1}(U)$ is an open set because $p$ is continuous, so $U$ is a $D$-open set of $\conti(T)$. 
  However, it is known that any smooth function $f\colon \conti(T)\to \mathbb{R}$ is locally constant (\CW, Example 3.15) so the connected sets are the only constituent of the $I$-topology of $\conti(T)$. 
  Therefore, the $D$-topology and the $I$-topology of $\conti(T)$ do not coincide if there is at least one open set that is not connected, this holds even if $T=\mathbb{R}$.
\end{remark}

The next example can be proved in the same manner as \exref{example cross right tangent} by using the global tangent version of \corref{corollary calculation of right tangent}.
However, we give some other ways of calculation by using properties of the global right tangent space.

\begin{ex} \label{example crossquot global external tangent}
  The diffeological space $\crossquot$ defined in \exref{example definition cross_quot generating family} is smoothly regular. 
  In particular, by \thref{theorem global right tangent space of regular space} and \thref{theorem global external tangent space of regular space}, there is an isomorphism of vector spaces, 
  \[
    \globalrighttangent{x}{\crossquot}=\globalexternal{x}{\crossquot}\cong
    \begin{cases}
      \mathbb{R}^2 & (x=(0,0))\\
      \mathbb{R} & (x\neq (0,0)).
    \end{cases}
  \]
\end{ex}

\begin{proof}
It is enough to show that $\crossquot$ is smoothly regular because we have already proved that
\[
  \righttangent{x}{\crossquot}=\external{x}{\crossquot}\cong
  \begin{cases}
    \mathbb{R}^2 & (x=(0,0))\\
    \mathbb{R} & (x\neq (0,0))
  \end{cases}
\]
in \exref{example cross right tangent}. We show that the $D$-topology and the $I$-topology of $\crossquot$ are the same. 
Let $U$ be a $D$-open set of $\crossquot$. The $D$-topology of $\crossquot$ is the quotient topology from $\mathbb{R}_{left}\amalg \mathbb{R}_{right}$ because of \IZ, 2.12.
Therefore, it is clear that $U$ is also an $I$-open set of $\crossquot$. This completes the proof.
\end{proof}

\begin{ex} \label{example crosssub global external tangent}
  There is an isomorphism of vector spaces, where $\crosssub$ is defined in \exref{example crosssub internal tangent}:
  \[
    \globalrighttangent{x}{\crosssub}=\globalexternal{x}{\crosssub}\cong
    \begin{cases}
      \mathbb{R}^2 & (x=(0,0))\\
      \mathbb{R} & (x\neq (0,0)).
    \end{cases}
  \]
\end{ex}

\begin{proof}
In \BKW, Example 5.1, it is proved that $\Pi \Phi(\crossquot)=\crosssub$. 
Since the equality $\Phi\Pi \Phi(\crossquot)=\Phi(\crossquot)$ holds (\proref{proposition reflexive stability}), we have $\Phi(\crossquot)=\Phi(\crosssub)$. 
Therefore, by \proref{proposition global right is differential invariant}, we have $\globalrighttangent{x}{\crosssub}\cong \globalrighttangent{x}{\crossquot}$. 
It also holds that $\globalexternal{x}{\crosssub}=\globalrighttangent{x}{\crosssub}$ 
because generators of $\globalrighttangent{x}{\crosssub}$ are in the image of the natural transformation $\nattransglobal_{(\crosssub,x)}$. 
\end{proof}

\begin{ex}[\CW, Example 3.24] \label{example orthogonal action quotient external tangent}
  We equip $\mathbb{R}^n/O(n)$ with the quotient diffeology from $\mathbb{R}^n$. 
  Let $0$ be the image of $0\in\mathbb{R}^n$ by the natural projection $\mathbb{R}^n\to \mathbb{R}^n/O(n)$.
  According to \CW, Example 3.24, there is an isomorphism of vector spaces:
  \[
    \righttangent{0}{\mathbb{R}^n/O(n)}=\external{0}{\mathbb{R}^n/O(n)}\cong \mathbb{R}.
  \]
  Moreover, $\mathbb{R}^n/O(n)$ is smoothly regular at $0$. Therefore, there is an isomorphism of vector spaces:
  \[
    \globalrighttangent{0}{\mathbb{R}^n/O(n)}=\globalexternal{0}{\mathbb{R}^n/O(n)}\cong \mathbb{R}.
  \]
\end{ex}

\begin{proof}
Let $U$ be a $D$-open neighborhood of $0$. $U$ is an open set of the quotient topology from $\mathbb{R}^n$ because of \IZ, 2.12.
Therefore, there is some $\varepsilon >0$ such that $[0,\varepsilon)\subset U$. Since $[0,\varepsilon)$ is an $I$-open neighborhood of $0$, it shows that $\mathbb{R}^n/O(n)$ is smoothly regular at $0$. 
\end{proof}

\begin{ex}[\CW, Example 3.23] \label{example irrational torus external tangent}
  Let $\theta\in\mathbb{R}\backslash\mathbb{Q}$, and let $\irrationaltorus$ be the irrational torus of slope $\theta$ defined in \exref{example irrational torus internal tangent}.
  Then for each $x\in \irrationaltorus$, there is an isomorphism of vector spaces:
  \[
    \globalrighttangent{x}{\irrationaltorus}=\globalexternal{x}{\irrationaltorus}\cong \righttangent{x}{\irrationaltorus}=\external{x}{\irrationaltorus}\cong 0.
  \]
  Moreover, $\irrationaltorus$ is smoothly regular.
\end{ex}

\begin{proof}
Since the $D$-topology of $\irrationaltorus$ is the indiscrete topology and any smooth function $\irrationaltorus\to\mathbb{R}$ is a constant function, 
the isomorphism is proved by \lemref{lemma locally smooth function globalexternal} and \lemref{lemma locally smooth function external}. 
Since the $D$-topology of $\irrationaltorus$ is the indiscrete topology, it is straightforward to see that $\irrationaltorus$ is smoothly regular at every point. 
\end{proof}

\begin{ex}[\CW, Example 3.22] \label{example wire diffeology external tangent}
  Let $\mathbb{R}^n_{wire}$ be the diffeological space $(\mathbb{R}^n, \mathscr{D}_{wire})$ defined in \exref{example definition wire diffeology}, and let $x\in \mathbb{R}_{wire}$. 
  According to \CW, Example 3.22, there is an isomorphism of vector spaces 
  (the external tangent space is determined only by 1-plots as in \CW, Proposition 3.13. The same proof is also applicable for the right tangent space):
  \[
    \righttangent{x}{\mathbb{R}_{wire}}=\external{x}{\mathbb{R}_{wire}}\cong \mathbb{R}^n.
  \]
  Moreover, $\mathbb{R}_{wire}$ is smoothly regular. 
  Then by \thref{theorem global right tangent space of regular space} and \thref{theorem global external tangent space of regular space}, there is an isomorphism of vector spaces:
  \[
    \globalrighttangent{x}{\mathbb{R}_{wire}}=\globalexternal{x}{\mathbb{R}_{wire}}\cong \mathbb{R}^n.
  \]
\end{ex}

\begin{proof}
We know that the $D$-topology of $\mathbb{R}^n_{wire}$ and the $D$-topology of $\mathbb{R}^n$ are the same because the $D$-topology is determined only by 1-plots (Christensen-Sinnamon-Wu \dtop, Theorem 3.7). 
Also, the differential spaces $\Phi(\mathbb{R}^n_{wire})$ and $\Phi(\mathbb{R}^n)$ are the same (\BKW, Example 2.8). 
From the above facts and the fact that $\mathbb{R}^n$ is smoothly regular, it follows that $\mathbb{R}^n_{wire}$ is also smoothly regular. 
\end{proof}

\begin{ex} \label{example extenal tangent horizontal quotient}
  Consider the pointwise smooth action of $\mathbb{R}$ on $\mathbb{R}^2$ given by $t\cdot (x,y)=(x+ty,y)$ for each $t\in \mathbb{R}$. 
  We write $X$ for the quotient diffeological space of $\mathbb{R}^2$ by this action (this is defined in \exref{example internal tangent horizontal quotient}). 
  Then $X$ is not smoothly regular.
  However, there is an isomorphism of vector spaces:
  \[
    \righttangent{[(0,0)]}{X}=\external{[(0,0)]}{X}\cong \mathbb{R},\; \globalrighttangent{[(0,0)]}{X}=\globalexternal{[(0,0)]}{X}\cong \mathbb{R}.
  \]
\end{ex}

\begin{proof}
  By \IZ, 2.12, 
  the set $\{[(x,0)]\in X\mid \text{$x\leq -1$ or $x\geq 1$}\}$ is a $D$-closed set. 
  Therefore, $U=X\backslash\{[(x,0)]\in X\mid \text{$x\leq -1$ or $x\geq 1$}\}$ is a $D$-open set.
  However, $U$ is not an $I$-open set. Indeed, any $I$-open set which contains the point $[(0,0)]$ also contains $[(2,0)]$ because for any smooth map $g\colon X\to \mathbb{R}$, $g([(0,0)])=g([(2,0)])$ holds. 
  Therefore, $X$ is not smoothly regular at $[(0,0)]$. 

  By the universality of the quotient diffeology, 
  there is an isomorphism of vector spaces:
  \begin{equation*}
    \begin{split}
    \smoothmap{X}{\mathbb{R}}&\cong \{f\colon \mathbb{R}^2\to \mathbb{R}\mid \text{$f$ is smooth and for any $(x, y)\in\mathbb{R}^2$ and $t\in\mathbb{R}$, $f(x,y)=f(x+ty,y)$}\}\\
    &\cong\{f\colon \mathbb{R}^2\to \mathbb{R}\mid \text{$f$ is smooth and for any $(x, y)\in\mathbb{R}^2$, $f(x,y)=f(0,y)$}\}\cong\smoothmap{Y}{\mathbb{R}},
    \end{split}
  \end{equation*}
  where $Y$ is the quotient diffeological space of $\mathbb{R}^2$ by the equivalence relation $\equivrel$ generated by $(x,y)\equivrel (0,y)$ for any $(x,y)\in\mathbb{R}$. 
  The second isomorphism holds because of the continuity of smooth maps in the set. 
  First, we prove that $\globalrighttangent{[(0,0)]}{X}\cong\mathbb{R}$. 
  By the above isomorphism, it is enough to show that $\globalrighttangent{[(0,0)]}{Y}\cong \mathbb{R}$. 
  However, since $Y$ is diffeomorphic to $\mathbb{R}$, we have $\globalrighttangent{[(0,0)]}{X}\cong\mathbb{R}$. 
  Moreover, one of nontrivial elements of $\globalrighttangent{x}{X}$ is $\nattransglobal_{(X,[(0,0)])}([p,\frac{d}{dx}])$, where $p\colon \mathbb{R}\to X;t\mapsto [(0,t)]$. 
  This is an element of $\globalexternal{[(0,0)]}{X}$ (\defref{definition natural transformation to global tangent space}). 
  Indeed, take a smooth map $g\colon X\to\mathbb{R};[(x,y)]\mapsto y$ (this is smooth because of the universality of the quotient diffeology), 
  then we have $g\circ p={\id}_{\mathbb{R}}$. This means that $\nattransglobal_{(X,[(0,0)])}([p,\frac{d}{dx}])$ is a nontrivial element of $\globalrighttangent{x}{X}$. 

  We also prove that $\righttangent{[(0,0)]}{X}\cong \mathbb{R}$. Take a smooth map $f\colon X\to \mathbb{R};[(x,y)]\mapsto y$ (this is smooth because of the universality of the quotient diffeology).
  It is enough to show that $f$ satisfies the conditions of \corref{corollary calculation of right tangent}. 
  Let $g\colon \coplotdom{g}\to \coplotim{g}$ be a smooth map from a $D$-open set $\coplotdom{g}$ of $X$ to a Euclidean open set $\coplotim{g}$. 
  Take a smooth map $h\colon\mathbb{R}\to\coplotim{g};t\mapsto g([0,t])$ (this is smooth because $g$ and the natural projection $\mathbb{R}^2\to X$ are smooth). 
  The equality $g=h\circ ({f|}_{\coplotdom{g}})$ holds because for any $[(x,0)]\in X$, $g([(x,y)])=g([(0,y)])$ holds. 
  Moreover, this $h$ is the unique map which satisfies $g=h\circ ({f|}_{\coplotdom{g}})$. Therefore, applying \corref{corollary calculation of right tangent}, we have $\righttangent{[(0,0)]}{X}\cong \mathbb{R}$.
  By the same argument as the case of the global tangent space, we obtain $\external{[(0,0)]}{X}\cong \mathbb{R}$.
\end{proof}

\begin{ex} \label{example right tangent not equal to global right tangent}
    Let $\equivrel$ be the equivalence relation on $\crosssub$ (defined in \exref{example crosssub internal tangent}) such that $s\equivrel t \Longleftrightarrow s=t$ or $s$ and $t$ are both elements of $\{(x,0)\in\crosssub\mid x\neq 0\}\subset\crosssub$. 
    We denote the quotient diffeological space $\crosssub/{\equivrel}$ by $X$. We denote the point represented by $(x,y)\in\crosssub$ by $[1,0]$ 
    Then, $X$ has the following properties:
    \begin{description}
        \item[1] $\righttangent{[1,0]}{X}=\external{[1,0]}{X}\cong 0$.
        \item[2] $X$ is not smoothly regular at $[1,0]$.
        \item[3] $\globalrighttangent{[1,0]}{X}=\globalexternal{[1,0]}{X}\cong \mathbb{R}$.
    \end{description}
\end{ex}

\begin{proof}
We first prove 1. By \IZ, 2.12, 
$\{[1,0]\}$ is a $D$-open set of $X$. 
Therefore, by using \corref{corollary subspace external tangent}, we have
\[
    \righttangent{[1,0]}{X}\cong\righttangent{[1,0]}{\{[1,0]\}}\cong 0.
\]
Since $\righttangent{[1,0]}{X}$ contains $\external{[1,0]}{X}$, we also have $\external{[1,0]}{X}\cong 0$. 

Next, we prove 3. By the universality of the quotient diffeology, 
there is an isomorphism of vector spaces:
\begin{equation*}
  \begin{split}
  \smoothmap{X}{\mathbb{R}}&\cong \{f\colon \crosssub\to \mathbb{R}\mid \text{$f$ is smooth and for any $x\in\mathbb{R}\backslash \{0\}$, $f(x,0)=f(1,0)$}\}\\
  &\cong\{f\colon \crosssub\to \mathbb{R}\mid \text{$f$ is smooth and for any $x\in\mathbb{R}$, $f(x,0)=f(1,0)$}\}\cong\smoothmap{Y}{\mathbb{R}},
  \end{split}
\end{equation*}
where $Y$ is the quotient diffeological space of $\crosssub$ by the equivalence relation $\equivrel$ generated by $(x,0)\equivrel (1,0)$ for any $x\in\mathbb{R}$. 
The second isomorphism holds because of the continuity of smooth maps in the set. 
Since $Y$ is diffeomorphic to $\mathbb{R}$, we have $\globalrighttangent{[1,0]}{X}\cong \globalrighttangent{0}{\mathbb{R}}\cong \mathbb{R}$. 
From this isomorphism, 
\[
    D\colon \smoothmap{X}{\mathbb{R}}\to\mathbb{R};f\mapsto \left.\frac{d}{dy}\right|_{y=0}f(0,y)
\]
is one of nontrivial elements in $\globalrighttangent{[1,0]}{X}$. This element is in $\globalexternal{[1,0]}{X}$ because of the \lemref{lemma differential operator is smooth}. 
Therefore, we have $\globalexternal{[1,0]}{X}\cong \mathbb{R}$. 
2 is proved by using 1, 3 and $\thref{theorem global right tangent space of regular space}$. 
\end{proof}

From here, we present some examples of which we have not been able to calculate the right tangent space or the external tangent space, 
but only we could calculate the Vincent's tangent space which is the image of the natural transformation from the internal tangent space (\remref{remark Vincent's tangent space}). 
We denote the image of $\nattrans_{(X,x)}\colon\internal{x}{X}\to \righttangent{x}{X}$ by $\vincenttangent{x}{X}$ 
and the image of $\nattransglobal_{(X,x)}\colon\internal{x}{X}\to \globalrighttangent{x}{X}$ by $\globalvincenttangent{x}{X}$. 

\begin{lemma} \label{lemma global Vincent tangent space of regular space}
  Let $X$ be a diffeological space, and let $x\in X$. If $X$ is smoothly regular at $x$, then there is an isomorphism of vector spaces:
  \[
    \globalvincenttangent{x}{X}\cong\vincenttangent{x}{X}.
  \]
  This isomorphism is induced from the isomorphism of \thref{theorem global right tangent space of regular space}.
\end{lemma}

\begin{proof}
  It is enough to show that $\alpha\circ\nattrans_{(X,x)}=\nattransglobal_{(X,x)}$ and $\beta\circ\nattransglobal_{(X,x)}=\nattrans_{(X,x)}$, 
  where $\alpha\colon\righttangent{x}{X}\to\globalrighttangent{x}{X}$ and $\beta\colon\globalrighttangent{x}{X}\to\righttangent{x}{X}$ are the maps defined in \thref{theorem global right tangent space of regular space}. 
  First, let $[p,v]\in\internal{x}{X}$, where $p\colon \plotdom{p}\to X$ is a plot of $X$ and $v\in T_0 (\plotdom{p})$. 
  For any $g\in\smoothmap{X}{\mathbb{R}}$, the following equality holds:
  \begin{equation*}
    \begin{split}
    &(\alpha\circ\nattrans_{(X,x)}([p,v]))(g)=(\nattrans_{(X,x)}([p,v]))([g,X])\\
    &=(T(g\circ {p|}_{p^{-1}(X)})v)[{\id}_{\mathbb{R}},\mathbb{R}]=(T(g\circ p)v)[{\id}_{\mathbb{R}},\mathbb{R}]=(\nattransglobal_{(X,x)}([p,v]))(g).
    \end{split}
  \end{equation*}
  Therefore, it follows that $\alpha\circ\nattrans_{(X,x)}=\nattransglobal_{(X,x)}$. 

  Second, let $[p,v]\in\internal{x}{X}$, where $p\colon \plotdom{p}\to X$ is a plot of $X$ and $v\in T_0 (\plotdom{p})$. 
  For any $g\in\smoothmap{\coplotdom{g}}{\mathbb{R}}$, where $\coplotdom{g}$ is a $D$-open set of $X$, take a smooth map $\tilde{g}$ and a $D$-open neighborhood $V$ of $x$ in \lemref{lemma existence of 0 extension dflg space}. 
  Then the following equality holds:
  \begin{equation*}
    \begin{split}
    &(\beta\circ\nattransglobal_{(X,x)}([p,v]))([g,\coplotdom{g}])=(\nattransglobal_{(X,x)}([p,v]))(\tilde{g})=(T(\tilde{g}\circ p)v)[{\id}_{\mathbb{R}},\mathbb{R}]\\
    &=(T(g\circ {p|}_{p^{-1}(V)})v)[{\id}_{\mathbb{R}},\mathbb{R}]=(\nattrans_{(X,x)}([p,v]))([{g|}_V,V])=(\nattrans_{(X,x)}([p,v]))([g,\coplotdom{g}]).
    \end{split}
  \end{equation*}
  Therefore, it follows that $\beta\circ\nattransglobal_{(X,x)}=\nattrans_{(X,x)}$ and this is what we wanted to prove.
\end{proof}

\begin{ex} \label{example n line gluing external tangent}
  Let $Y_{A}=(\coprod_{\alpha\in A}\mathbb{R}_{\alpha})/{\equivrel}$ be the diffeological space defined in \exref{example n line gluing internal tangent}. 
  Then $Y_A$ is smoothly regular (this is straightforward). 
  Also, there is an isomorphism of vector spaces:
  \[
    \vincenttangent{[0]}{Y_{A}}\cong\globalvincenttangent{[0]}{Y_A}\cong\mathbb{R}.
  \]
\end{ex}

\begin{proof}
We have already proved that $\internal{[0]}{Y_A}\cong \bigoplus_{\alpha\in A}\mathbb{R}$ 
and $\internal{[0]}{Y_A}$ is generated by $\{[i_{\alpha},\frac{d}{dx}]\}_{\alpha\in A}$ where $i_{\alpha}\colon\mathbb{R}_{\alpha}\to Y_A$ is the natural inclusion in \exref{example n line gluing internal tangent}. 
Therefore, it is enough to show that the subspace of $\globalrighttangent{[0]}{Y_A}$ generated by $\{\nattransglobal_{(Y_A,[0])}([i_{\alpha},\frac{d}{dx}])\}_{\alpha\in A}$ is $1$-dimensional. 
In particular, we show that for any $\alpha$, $\beta\in A$, the equality $\nattransglobal_{(Y_A,[0])}([i_{\alpha},\frac{d}{dx}])=\nattransglobal_{(Y_A,[0])}\beta([i_{\beta},\frac{d}{dx}])$ holds. 
For any $g\in\smoothmap{Y_A}{\mathbb{R}}$, and any $\alpha\in A$, 
the following equality holds:
\[
  \left(\nattransglobal_{(Y_A,[0])}\left(\left[i_{\alpha},\frac{d}{dx}\right]\right)\right)(g)=\left(T\left(g\circ i_{\alpha}\right)\frac{d}{dx}\right)[{\id}_{\mathbb{R}},\mathbb{R}]=\left.\frac{d}{dx}\right|_{x=0}g\circ i_{\alpha}.
\]
However, for any $x\leq 0$ and $\alpha$, $\beta\in A$, $g\circ i_{\alpha}(x)=g\circ i_{\beta}(x)$ holds because of the equality $i_{\alpha}(x)=i_{\beta}(x)$ for $x\leq 0$. 
Therefore, we have the next equality:
\[
  \left.\frac{d}{dx}\right|_{x=0}g\circ i_{\alpha}=\left.\frac{d}{dx}\right|_{x=0}g\circ i_{\beta}.
\]
This means that for any $\alpha$, $\beta\in A$, $\nattransglobal_{(Y_A,[0])}([i_{\alpha},\frac{d}{dx}])=\nattransglobal_{(Y_A,[0])}\beta([i_{\beta},\frac{d}{dx}])$. 
\end{proof}
  I have not known whether $\righttangent{[0]}{Y_A}$ has some elements other than those of \exref{example n line gluing external tangent} or not.

\begin{ex} \label{example external tangent eigenvalue>,<1}
  Consider the pointwise smooth action of $\mathbb{R}$ on $\mathbb{R}^2$ given by $t\cdot (x,y)=(2^tx, 2^{-t}y)$ for each $t\in \mathbb{R}$. 
  We write $X$ for the quotient diffeological space of $\mathbb{R}^2$ by this action. 
  The right tangent space of $X$ is isomorphic to the right tangent space of $Y$, where $Y$ is a quotient diffeological space of $\mathbb{R}^2$ by the equivalence relation $\equivrel$ 
  generated by the relation induced from the action $\varphi$ and the relation $(s,0)\equivrel (0,0)$, $(0,s)\equivrel (0,0)$ for each $s\in \mathbb{R}$ (the same reason as \exref{example extenal tangent horizontal quotient}). 
  Also, let $\cross$ be the set $\cross=\{(x,y)\in\mathbb{R}\mid \text{$x=0$ or $y=0$}\}$. 
  Take four maps $\mathbb{R}\to \cross$ as follows: $\alpha_1(t)=\begin{cases}
    (t,0) & (t\leq 0)\\
    (0,t) & (t\geq 0)
  \end{cases}$, $\alpha_2(t)=\begin{cases}
    (-t,0) & (t\leq 0)\\
    (0,t) & (t\geq 0)
  \end{cases}$, $\alpha_3(t)=\begin{cases}
    (-t,0) & (t\leq 0)\\
    (0,-t) & (t\geq 0)
  \end{cases}$, $\alpha_4(t)=\begin{cases}
    (t,0) & (t\leq 0)\\
    (0,-t) & (t\geq 0)
  \end{cases}$. We equip the set $\cross$ with the diffeology generated by $\{\alpha_1,\alpha_2,\alpha_3,\alpha_4\}$ and we denote this diffeological space by $Z$. 
  Then, $X$, $Y$ and $Z$ have the following properties:
  \begin{description}
    \item[1] $Y$ is diffeomorphic to $Z$. 
    \item[2] $Y$ and $Z$ are smoothly regular.
    \item[3] $\internal{(0,0)}{Z}\cong \mathbb{R}^4$.
    \item[4] $\globalvincenttangent{(0,0)}{Z}\cong \mathbb{R}$ (remark that $\globalvincenttangent{[(0,0)]}{X}\cong 0$ because we proved that $\internal{[(0,0)]}{X}\cong 0$ in \exref{example internal tangent eigenvalue>,<1}). 
    \item[5] $\globalrighttangent{(0,0)}{X}$ and $\globalexternal{(0,0)}{X}$ are at least $1$-dimensional.
  \end{description}
\end{ex}

\begin{proof}
  We first prove 1. We define four $D$-open sets of $Y$ as $U_1^{+}=\{[(x,y)]\mid x>0\}$, $U_1^{-}=\{[(x,y)]\mid x<0\}$, $U_2^{+}=\{[(x,y)]\mid y>0\}$, $U_2^{-}=\{[(x,y)]\mid y<0\}$, 
  which satisfies $Y=U_1^{+}\cup U_1^{-}\cup U_2^{+}\cup U_2^{-}$. 
  Take four smooth maps 
  \[
    \eta_1^{+}\colon U_1^{+}\to Z;[(x,y)]\mapsto\begin{cases}
      (xy,0) & (y\leq 0)\\
      (0,xy) & (y\geq 0)
    \end{cases},\; 
    \eta_2^{+}\colon U_2^{+}\to Z;[(x,y)]\mapsto\begin{cases}
      (-xy,0) & (x\leq 0)\\
      (0,xy) & (x\geq 0),
    \end{cases}
  \]
  \[
    \eta_1^{-}\colon U_1^{-}\to Z;[(x,y)]\mapsto\begin{cases}
      (0,-xy) & (y\leq 0)\\
      (-xy,0) & (y\geq 0)
    \end{cases},\; 
    \eta_2^{-}\colon U_2^{-}\to Z;[(x,y)]\mapsto\begin{cases}
      (0,-xy) & (x\leq 0)\\
      (xy,0) & (x\geq 0).
    \end{cases}
  \]
  By combining these smooth maps, we define $\eta\colon Y\to Z$ (these maps coincide within the intersections of their domains). 
  Moreover, take a smooth map 
  \[
    \gamma\colon Z\to Y;(x,y)\mapsto \begin{cases}
      [(1,x)] & (x\leq 0, y=0)\\
      [(1,y)] & (y\geq 0, x=0)\\
      [(-1,x)] & (x\geq 0, y=0)\\
      [(-1,y)] & (y\leq 0, x=0).
    \end{cases}
  \]
  This is smooth because for each $i\in\{1,2,3,4\}$, the composition $\gamma\circ\alpha_{i}$ is smooth. Indeed, 
  \[
    \gamma\circ \alpha_1(t)=[(1,t)]\text{ and } \gamma\circ \alpha_3(t)=[(-1,-t)]
  \]
  hold. Also, the next two equalities 
  \[
    \gamma\circ \alpha_2(t)
    =\begin{cases}
      [(-1,-t)] & (t\leq 0)\\
      [(1,t)] & (t\geq 0)
    \end{cases}=[(t,1)]\text{ and }
    \gamma\circ \alpha_4(t)
    =\begin{cases}
      [(1,t)] & (t\leq 0)\\
      [(-1,-t)] & (t\geq 0)
    \end{cases}=[(-t,-1)]
  \]
  hold because for any $t<0$, $[(-1,-t)]=[(2^{\log_2 (-t)}\cdot (-1),2^{-\log_2 (-t)}\cdot (-t))]=[(t,1)]$ and $[(1,t)]=[(2^{\log_2 (-t)}\cdot 1,2^{-\log_2 (-t)}\cdot t)]=[(-t,-1)]$ hold, 
  and for any $t>0$, $[(1,t)]=[(2^{\log_2 t}\cdot 1,2^{-\log_2 t}\cdot t)]=[(t,1)]$ and $[(-1,-t)]=[(2^{\log_2 t}\cdot (-1),2^{-\log_2 t}\cdot (-t))]=[(-t,-1)]$ hold. 
  Since it is straightforward to check that $\eta$ and $\gamma$ are inverses of each other, it follows that $Y$ and $Z$ are diffeomorphic.

  Next, we prove 2. Since $Y$ and $Z$ are diffeomorphic to each other (by 1), it is enough to show that $Z$ is smoothly regular. 
  Since $Z$ is also a quotient diffeological space of $\mathbb{R}\coprod\mathbb{R}\coprod\mathbb{R}\coprod\mathbb{R}$, 
  the $D$-open set of $Z$ is a set $U\subset Z$ such that $\pi^{-1}(U)$ is open,where $\pi\colon \mathbb{R}\coprod\mathbb{R}\coprod\mathbb{R}\coprod\mathbb{R}\to Z$ is the natural projection. 
  Therefore, it is clear to prove that $U$ is also an $I$-open set. Namely, $Z$ is smoothly regular. 

  From here, we prove 3. It is enough to show that $\{[\alpha_{i},\frac{d}{dx}]\}_{i=1}^4$ is a basis of $\internal{(0,0)}{Z}$. 
  We can prove this in the same manner as \exref{example n line gluing internal tangent}. 

  We prove 4. In the proof of 3, we have shown that $\{[\alpha_{i},\frac{d}{dx}]\}_{i=1}^4$ is a basis of $\internal{(0,0)}{Z}$. 
  Therefore, it is enough to show that the image of $[\alpha_{i},\frac{d}{dx}]$ by $\nattransglobal_{(Z,(0,0))}$ is the same element for each $i\in\{1,2,3,4\}$. 
  As in the proof of \exref{example n line gluing external tangent}, it is enough to show that for any $g\in \smoothmap{Z}{\mathbb{R}}$, 
  the value $\left.\frac{d}{dx}\right|_{t=0} g\circ \alpha_i$ is independent to the choice of $i\in\{1,2,3,4\}$. 
  Since $g\circ \alpha_1=g\circ \alpha_2$ for $t\geq 0$, we have $\left.\frac{d}{dx}\right|_{t=0} g\circ \alpha_1=\left.\frac{d}{dx}\right|_{t=0} g\circ \alpha_2$ (the same as the proof of \exref{example n line gluing external tangent}). 
  The values $\left.\frac{d}{dx}\right|_{t=0} g\circ \alpha_3$ and $\left.\frac{d}{dx}\right|_{t=0} g\circ \alpha_4$ are the same values by the same argument. 

  We prove 5. From 1 and 4, we have $\globalrighttangent{(0,0)}{Y}\supset \globalvincenttangent{(0,0)}{Y}\cong \mathbb{R}$. Therefore, $\globalrighttangent{(0,0)}{Y}$ is at least $1$-dimensional. 
  Also, by the same argument as \exref{example extenal tangent horizontal quotient}, we have $\globalrighttangent{(0,0)}{X}\cong \globalrighttangent{(0,0)}{Y}$. 
  Therefore, $\globalrighttangent{(0,0)}{Y}$ is also at least $1$-dimensional. 
\end{proof}

\begin{ex} \label{example crushing y-axis external tangent}
  Let $\equivrel$ be the equivalence relation on $\mathbb{R}$ such that $x\equivrel y\Longleftrightarrow x=y$ or $x$ and $y$ are both elements of $\{(x,y)\in \mathbb{R}^2\mid x=0\}$.
  We denote the quotient diffeological space $\mathbb{R}^2/\equivrel$ by $\mathbb{R}^2/\{x=0\}$. 
  Then $\globalvincenttangent{[(0,0)]}{\mathbb{R}^2/\{x=0\}}$ and $\vincenttangent{[(0,0)]}{\mathbb{R}^2/\{x=0\}}$ are uncountably-infinite dimensional.
\end{ex}

\begin{proof}
  Just like the proof of \exref{example crushing y-axis internal tangent}, 
  we can prove that all $\globalrightfunctor(p_t)(\frac{d}{dx})\in\globalrighttangent{[(0,0)]}{\mathbb{R}^2/\{x=0\}}$ for each $t\in \mathbb{R}$ 
  and all $\rightfunctor(p_t)(\frac{d}{dx})\in\righttangent{[(0,0)]}{\mathbb{R}^2/\{x=0\}}$ for each $t\in \mathbb{R}$ are linearly independent respectively, 
  where $\frac{d}{dx}$ is one of the non trivial elements of $T_0(\mathbb{R})\cong \globalrighttangent{0}{\mathbb{R}}\cong \righttangent{0}{\mathbb{R}}$.
\end{proof}

\appendix
\section{Differential spaces and \frolicher spaces} \label{section differential and frolicher}

This appendix provides an overview of the definitions and properties of differential spaces and \frolicher spaces.  
The definitions and properties presented here are found in 
\watts, and \BKW, with the addition of 
\proref{proposition the limit of reflexive dflg}, which was not stated in \watts, and \BKW~explicitly. 

\begin{defi}[\watts, Definition 2.17] \label{definition differential space}
  Let $X$ be a set. A \italic{differential structure} on $X$ is a family of real-valued functions $\mathscr{F}_X$ on $X$, such that the following two axioms are satisfied:
  \begin{description}
    \item[Smooth compatibility] For any positive integer $k\in\mathbb{Z}$, functions $f_1,\dots,f_k\in\mathscr{F}_X$, and $F\in C^{\infty}(\mathbb{R}^k,\mathbb{R})$, 
    the composition $F(f_1,\dots,f_k)$ is in $\mathscr{F}_X$. 
    \item[Locality] Let $f\colon X\to\mathbb{R}$ be a function such that for any $x\in X$, there exists an $I$-open neighborhood $U\subset X$ of $x$ 
    and a function $g\in \mathscr{F}_X$ satisfying ${f|}_U={g|}_U$. Then $f\in \mathscr{F}_X$, 
  \end{description}
  where we equip $X$ with the initial topology of $\mathscr{F}_X$, called the \italic{$I$-topology} and we call open sets of the $I$-topology \italic{$I$-open sets}. 
  A \italic{differential space} is a set equipped with a differential structure. We usually write the underlying set $X$ to represent the differential space $(X,\mathscr{F}_X)$.
\end{defi}

\begin{defi}[\watts, Definition 2.20] \label{definition smooth map between differential spaces}
  Let $(X,\mathscr{F}_X)$ and $(Y,\mathscr{F}_Y)$ be two differential spaces. A map $f\colon X\to Y$ is said to be smooth if for every element $g\in\mathscr{F}_Y$, 
  the composition map $g\circ f$ is in $\mathscr{F}_X$. Differential spaces and smooth maps between them form a category denoted by $\dfrn$. We call an isomorphism of this category a \italic{diffeomorphism}. 
  The category $\dfrn$ contains the category $\mfd$ as a full subcategory as follows.
\end{defi}

\begin{ex} \label{example manifold as differential space}
  Let $M$ be a smooth manifold, and let $\mathscr{F}_M$ be the set of all smooth maps $M\to \mathbb{R}$. Then $(M, \mathscr{F}_M)$ is a differential space. 
  The set of all smooth maps between two manifolds coincide with the set of all smooth maps regarding the two manifolds as differential spaces. 
\end{ex}

We can also define the limit and colimit of any small diagram in $\dfrn$. See \watts, 2.2.

There are adjoint functors between the categories $\dflg$ and $\dfrn$.
\begin{defi}[\watts, Lemma 2.42 and Lemma 2.44] \label{definition adjoint between dflg and dfrn}
  Let $(X,\mathscr{D}_X)$ be a diffeological space. 
  We define the set 
  \[
    \Phi \mathscr{D}_X=\{f\colon X\to\mathbb{R}\mid \text{For any $p\in\mathscr{D}_X$, $f\circ p\in C^{\infty}(\plotdom{p},\mathbb{R})$.}\}.
  \]
  Then $(X,\Phi \mathscr{D}_X)$ is a differential space. If $(X,\mathscr{D}_X)$ and $(Y,\mathscr{D}_Y)$ are two diffeological spaces and $f\colon X\to Y$ is smooth as a map between diffeological spaces,  
  then the map $f$ is also smooth as a map between differential spaces $(X,\Phi\mathscr{D}_X)$ and $(Y,\Phi\mathscr{D}_Y)$. Therefore, we have the functor $\Phi\colon \dflg\to\dfrn$.
  
  Also, let $(X,\mathscr{F}_X)$ be a differential space. We define the set 
  \[
    \Pi \mathscr{F}_X=\{p\colon \plotdom{p}\to X\mid \text{$p$ is a parametrization of $X$ and for any $f\in\mathscr{F}_X$, $f\circ p\in C^{\infty}(\plotdom{p},\mathbb{R})$.}\}.
  \]
  Then $(X,\Pi\mathscr{F}_X)$ is a diffeological space. If $(X,\mathscr{F}_X)$ and $(Y,\mathscr{F}_Y)$ are two differential spaces and $f\colon X\to Y$ is smooth as a map between differential spaces,  
  then the map $f$ is also smooth as a map between differential spaces $(X,\Pi\mathscr{F}_X)$ and $(Y,\Pi\mathscr{F}_Y)$. Therefore, we have the functor $\Pi\colon \dfrn\to\dflg$.
  The functor $\Phi$ is the left adjoint functor of $\Pi$. 
\end{defi}

\begin{remark}{}{} \label{remark I-topology for dflg sp}
  If $X$ is a diffeological space, then the $I$-topology of $\Phi X$ is simply called the $I$-topology of $X$. 
  Also, if $X$ is a differential space, then the $D$-topology of $\Pi X$ is simply called the $D$-topology of $X$. 
\end{remark}

\begin{defi}[\BKW, Definition 2.6] \label{definition reflexive}
  A diffeology $\mathscr{D}$ is said to be \italic{reflexive} if $\mathscr{D}=\Pi\Phi\mathscr{D}$. A differential structure $\mathscr{F}$ is said to be reflexive if $\mathscr{F}=\Phi\Pi\mathscr{F}$.  
\end{defi}

\begin{prop}[\BKW, Proposition 2.7] \label{proposition reflexive stability}
  Let $X$ be a set. For any differential structure $\mathscr{F}$ on $X$, $\Pi(\mathscr{F})$ is a reflexive diffeology on $X$. 
  For any diffeology $\mathscr{D}$ on $X$, $\Phi(\mathscr{D})$ is a reflexive differential structure on $X$. 
\end{prop}

\begin{prop}[\BKW, Theorem 2.13] \label{proposition reflexive as frolicher}
  The full subcategory of $\dflg$ whose objects are all reflexive diffeological spaces and the full subcategory of $\dfrn$ whose objects are all reflexive differential spaces 
  are category equivalent to the category of \frolicher space and smooth maps between them, defined in \BKW, Definition 2.12.
\end{prop}

Manifolds equipped with the standard diffeology are examples of reflexive diffeological spaces. 
Also, the following proposition holds. 

\begin{prop} \label{proposition the limit of reflexive dflg}
  We denote the diffeology of a diffeological space $X$ by $\mathscr{D}_{X}$. The following two properties hold:
  \begin{description}
    \item[1] Let $\{X_{\alpha}\}_{\alpha\in A}$ be a family of reflexive diffeological spaces. 
    Then the product diffeological space $\prod_{\alpha\in A} X_{\alpha}$ is reflexive. 
    \item[2] Let $X$ be a diffeological space, and let $Y$ be a subset of $X$. If we equip $Y$ with the subdiffeology of $X$, then $Y$ is reflexive. 
  \end{description}
\end{prop}

\begin{proof}
We first prove 1. It is enough to show that each plot $p\in \Pi \Phi(\mathscr{D}_{\prod_{\alpha\in A} X_{\alpha}})$ is also in $\mathscr{D}_{\prod_{\alpha\in A} X_{\alpha}}$. 
To prove this, we show that for any $\alpha\in A$, the composition $p_{\alpha}\circ p$ is in $\mathscr{D}_{X_{\alpha}}$, where $p_{\alpha}\colon {\prod_{\alpha\in A} X_{\alpha}}\to X_{\alpha}$ is the natural projection. 
By the functionality of $\Pi \Phi$, the map $p_{\alpha}\colon (\prod_{\alpha\in A} X_{\alpha},\Pi \Phi(\mathscr{D}_{\prod_{\alpha\in A} X_{\alpha}}))\to (X_{\alpha},\Pi \Phi\mathscr{D}_{X_{\alpha}})$ is smooth. 
This shows that $p_{\alpha}\circ p$ is in $\Pi \Phi\mathscr{D}_{X_{\alpha}}$. Since $X_{\alpha}$ is reflexive, $p_{\alpha}\circ p$ is in $\mathscr{D}_{X_{\alpha}}$. 
Therefore, $p$ is in $\mathscr{D}_{\prod_{\alpha\in A} X_{\alpha}}$. 

Second, we prove 2. It is enough to show that each plot $p\in \Pi \Phi\mathscr{D}_{Y}$ is also in $\mathscr{D}_Y$. 
To prove this, we prove that $i\circ p$ is in $\mathscr{D}_X$, where $i\colon Y\to X$ is the natural inclusion. 
In the same manner as the proof of 1, we have $i\circ p\in \Pi \Phi(\mathscr{D}_{X})$. Since $X$ is reflexive, this means that $i\circ p\in \mathscr{D}_X$. 
\end{proof}

\begin{remark}{}{}
  We can also prove that any limit of reflexive diffeological spaces as diffeological spaces are reflexive diffeological space in the same manner as this proposition. 
  By this proposition, both $[0,\infty)$ with the subdiffeology from $\mathbb{R}$ and $\crosssub$ in \exref{example crosssub internal tangent} are reflexive diffeological spaces.
  
  However, quotient diffeological spaces or sum diffeological spaces may not be reflexive. 
  For example, neither $\mathbb{R}/O(n)$ with the quotient diffeology from $\mathbb{R}$ nor $\crossquot$ in \exref{example definition cross_quot generating family} is a reflexive diffeological space. 
  Indeed, $\Pi \Phi(\mathbb{R}/O(n))=[0,\infty)$ for any $n\in \mathbb{Z}_{\geq 0}$ and $\Pi \Phi(\crossquot)=\crosssub$ hold (\BKW, Example 4.1, Example 5.1). 
  We have already shown that $\mathscr{D}_{\mathbb{R}/O(n)}\subsetneq \mathscr{D}_{[0,\infty)}$ and $\mathscr{D}_{\crossquot}\subsetneq\mathscr{D}_{\crosssub}$ 
  in \remref{remark half line is not diffeo to orthogonal quotient} and \remref{remark crosssub is not diffeo to crossquot}. 
  On the other hand, if we take a colimit of \frolicher spaces regarding them as differential spaces, then the colimit differential space is also a reflexive differential space. 
  The proof of this claim is the same as that of \proref{proposition the limit of reflexive dflg}.
\end{remark}

\end{document}